\documentclass[11 pt]{amsart}
\usepackage{hyperref}
\usepackage{changebar}
\usepackage{latexsym}
\usepackage{amssymb, amsmath}
\usepackage{amsthm}
\usepackage{geometry}
\usepackage{comment}
\usepackage[T1]{fontenc}
\usepackage{lmodern}

\usepackage[stretch=10]{microtype}
\usepackage{tikz}
\usetikzlibrary{calc}
\usetikzlibrary{decorations.pathreplacing,angles,quotes}
\usepackage{fixltx2e}

\numberwithin{equation}{section}

\newtheorem{theorem}{Theorem}[section]
\newtheorem{lemma}[theorem]{Lemma}
\newtheorem{cor}[theorem]{Corollary}
\newtheorem{sublem}[theorem]{Sublemma}
\newtheorem{proposition}[theorem]{Proposition}

\newtheorem{defin}[theorem]{Definition}

\newtheorem{remark}[theorem]{Remark}
\newtheorem{convention}[theorem]{Convention}

\newcommand{\cA}{\mathcal{A}}
\newcommand\cB{{\mathcal B}}

\newcommand{\cC}{\mathcal{C}}

\newcommand{\cE}{\mathcal{E}}
\newcommand{\cF}{\mathcal{F}}
\newcommand{\cG}{\mathcal{G}}
\newcommand{\cI}{\mathcal{I}}

\newcommand{\cK}{\mathcal{K}}
\newcommand{\cL}{\mathcal{L}}

\newcommand{\cM}{\mathcal{M}}
\newcommand{\cN}{\mathcal{N}}
\newcommand{\cP}{\mathcal{P}}
\newcommand{\hP}{\mathring{\cP}}
\newcommand{\cQ}{\mathcal{Q}}
\newcommand\cR{{\mathcal R}}
\newcommand{\cS}{\mathcal{S}}
\newcommand{\cV}{\mathcal{V}}
\newcommand{\cW}{\mathcal{W}}
\newcommand{\cT}{\mathcal{T}}

\newcommand{\cZ}{\mathcal{Z}}

\newcommand{\bV}{\mathbb{V}}

\newcommand{\bpsi}{\overline{\psi}}

\newcommand{\tpsi}{\widetilde{\psi}}
\newcommand{\tnu}{\tilde\nu}

\newcommand{\hW}{\widehat{\cW}}

\newcommand{\htop}{h_{\scriptsize{\mbox{top}}}}
\newcommand{\musrb}{\mu_{\tiny{\mbox{SRB}}}}
\newcommand{\fsrb}{f_{\tiny{\mbox{SRB}}}}
\newcommand{\hatmusrb}{\hat \mu_{\tiny{\mbox{SRB}}}}
\newcommand{\mutop}{\mu_{\tiny{\mbox{top}}}}

\newcommand{\vf}{\varphi}

\newcommand{\ve}{\varepsilon}

\newcommand{\ds}{\displaystyle}

\def\beq{\begin{equation}}
\def\eeq{\end{equation}}

\newcommand{\diam}{\mbox{diam}}

\begin{document}

\title[Measure of Maximal Entropy for Piecewise Hyperbolic Maps]{Uniqueness and Exponential Mixing for the Measure of Maximal Entropy for Piecewise Hyperbolic Maps}
\author{Mark F. Demers}
\address{Department of Mathematics, Fairfield University, Fairfield CT 06824, USA}
\email{mdemers@fairfield.edu}

\thanks{
MD was partly supported by NSF grant DMS 1800321.}

\date{\today}

\begin{abstract}
For a class of piecewise hyperbolic maps in two dimensions, we propose a combinatorial definition
of topological entropy by counting the maximal, open, connected components of the phase
space on which iterates of the map
are smooth.  We prove that this quantity dominates the measure theoretic entropies of all
invariant probability measures of the system, and then construct
an invariant measure   
whose entropy equals the proposed topological entropy.  We prove that our
measure is the unique measure of maximal entropy, that it is ergodic, gives positive measure
to every open set, and
has exponential decay of
correlations against H\"older continuous functions.  As a consequence, we also prove a 
lower bound on the
rate of growth of periodic orbits. The main tool used in the paper
is the construction of anisotropic Banach spaces of distributions on which the
relevant weighted transfer operator has a spectral gap.  We then construct our measure
of maximal entropy by taking a product of left and right maximal eigenvectors of this
operator.
\end{abstract}

\maketitle


\section{Introduction}
\label{sec:intro}

There has been a flurry of recent activity in establishing the existence and uniqueness
of equilibrium states for broad classes of potentials and systems outside the uniformly hyperbolic
setting.  This topic traces back to the work of Margulis \cite{Mar0}, who proved that the
number of periodic orbits of length at most $L$ for the geodesic flow on a compact manifold of
strictly negative curvature grow at an exponential rate determined by the topological entropy
$\htop$ of
the flow.  To prove this result, Margulis constructed an invariant measure
$\mutop$ via conditional measures on the local stable and unstable manifolds of the flow
which scaled by $e^{\pm t \htop}$.  An important feature of the measure
$\mutop$ is that it is the unique measure of maximal entropy for the flow:
its measure-theoretic entropy equals the topological entropy $\htop$.

These results were generalized and further developed for broader classes of Anosov and Axiom A flows and diffeomorphisms
through the work of Sinai, Bowen, Ruelle and many others using thermodynamic formalism 
\cite{Sin72, BoRu, ruelle1}, topological techniques \cite{Bo1,Bo2,Bo3,Bo4},
and dynamical zeta functions \cite{PaPo83, ruelle2}.  Later, Dolgopyat's proof of exponential
decay of correlations for some geodesic flows \cite{Do} led to more precise asymptotics
for counting periodic orbits \cite{PS}.

Recent attempts to extend proofs of existence and uniqueness of equilibrium states
in general, and measures of maximal entropy in particular, to the nonuniformly hyperbolic setting
have employed symbolic dynamics \cite{Sar1, Sar2, lima, BuSa}, as well as adapting the approach of Bowen via a notion of non-uniform specification
\cite{BCFT, CFT, CliWar, CPZ}.  These works have greatly broadened the classes of systems for which
one can prove the existence and uniqueness of equilibrium states, yet they do not usually provide
rates of mixing for these measures.

Simultaneously, there have been advances made in the study of the transfer operator
associated with hyperbolic systems
with singularities, first to piecewise hyperbolic maps (with bounded derivatives)
\cite{demers liv, baladi gouezel1, baladi gouezel2}, and then to 
dispersing and other hyperbolic billiards \cite{demzhang11, demzhang13, demzhang14}.
This approach, which avoids the coding associated with Markov partitions or extensions,
exploits the hyperbolicity of the system to prove that the action of the transfer operator
on appropriately defined Banach spaces has good spectral properties.  It was used recently
to prove exponential decay of correlations for the finite horizon Sinai billiard flow 
\cite{bdl}, adapting ideas of Dolgopyat \cite{Do} and Liverani \cite{Li2}.
It was then applied to prove the existence and uniqueness of a measure of
maximal entropy for finite horizon Sinai billiard maps \cite{max}, establishing
a variational principle for this class of billiards.

For hyperbolic systems with discontinuities, a priori results that guarantee
the existence of an invariant  measure maximizing the entropy, or even
a simple definition of topological
entropy, are not available as they are for continuous maps and flows. Indeed, in
order to overcome this shortcoming, one approach is to redefine the map as a continuous
map on a noncompact space, and then apply generalized definitions of topological entropy
in this setting.  Yet such definitions can be cumbersome to work with, and the resulting entropy
can depend on the choice of metric in the reduced space.  

To simplify matters, the first step in \cite{max} is to define
an intuitive notion of growth in complexity given by the number of domains of
continuity $\cM_0^n$ for the map  $T^n$.  This leads to an asymptotic quantity $h_*$, which plays the
role of topological entropy \cite[Definition~2.1]{max} (see also Definition~\ref{def:h_*} below).  This quantity
is proved to equal the supremum of the measure-theoretic entropies of the invariant measures for the
billiard map, and the unique measure $\mu_*$ whose entropy achieves this maximum is constructed
by taking a product of left and right maximal eigenvectors of an associated weighted transfer operator $\cL$,
following the methods in \cite{GL2} which generalize the classical Parry construction.

Despite this success, the weight in the relevant transfer operator in \cite{max} is unbounded
due to the unbounded expansion and contraction that occur near grazing collisions in
dispersing billiards.  The presence of this weight forced significant changes in the Banach spaces
from \cite{demzhang11} on which 
the operator acted, and it was not possible to establish a spectral gap in this context.  Indeed,
the rate of mixing for the measure of maximal entropy is an open question for billiards.

The purpose of the present paper is to demonstrate that under the additional assumption
that the derivative of the map is bounded, the techniques employed in \cite{max}
are sufficient to prove the existence and uniqueness of a measure of maximal entropy that
is exponentially mixing.
To this end, we study a class of piecewise hyperbolic maps, defined in Section~\ref{setting}.
The existence and statistical properties of Sinai-Ruelle-Bowen (SRB) measures\footnote{Recall that an SRB measure for a hyperbolic system is an invariant probability measure whose conditional
measures on local unstable manifolds are absolutely continuous with respect to the Riemannian
volume.} for
this class of maps has been studied via a variety of techniques
\cite{Pesin1, Li1, Y98, demers liv, baladi gouezel2}.
Transfer operators with more general potentials were studied in \cite{baladi gouezel2} and
a bound on the essential spectral radius was obtained; however, lower bounds on the spectral
radius of the transfer operator were not obtained, so that no spectral gap was established
and the related invariant measures were not constructed. 
Currently there are no results regarding measures
of maximal entropy, nor more general equilibrium states for this class of maps.

In structure, this paper mainly follows the approach in \cite{max}.  Yet there are several
key differences between the class of piecewise hyperbolic maps studied here and 
dispersing billiards.  The primary simplification is that our maps have bounded derivatives,
as mentioned above,
and this fact permits us to prove a spectral gap for the relevant transfer operator, which
leads to exponential decay of correlations for the measure of maximal entropy $\mu_*$.
However, there are two additional difficulties in the current setting that are not present in
Sinai billiards.
\begin{itemize}
  \item[(i)]  We do not assume that the singularity curves for our map $T$ satisfy the
  {\em continuation of singularities} property enjoyed by billiards.
  \item[(ii)]  We do not assume the map is associated with a continuous flow.
\end{itemize}
Point (i) creates significant complications in the study of the rate of growth of
$\# \cM_0^n$, the number of maximal, connected domains of continuity of $T^n$.  In particular,
the submultiplicative property of $\# \cM_0^n$ proved in \cite[Lemma~3.3]{max}, and often
exploited in that work, may fail in the present context due to the fact that dynamical refinements
of $\# \cM_0^n$ may have elements that are not simply connected.  
Indeed, the uniform exponential upper and lower
bounds on $\# \cM_0^n$ proved in Proposition~\ref{prop:M0n} are completed only after the
spectral gap for the operator $\cL$ is established.
Point (ii) has several consequences. The first is that the continuous flow provides a linear
bound on the growth in complexity, which is exploited in \cite{max}.  In the present work,
this property is replaced by the complexity assumption (P1) introduced in Section~\ref{sec:pw}; 
while the growth in complexity
may be exponential for our class of maps, it is slow relative to the minimum hyperbolicity constant
for the map (and therefore also relative to $h_*$ by Lemma~\ref{lem:growth}(d)).  
The second consequence of (ii) is that in \cite{max}, there is a positive minimum distance
between orbits that belong to different elements of $\cM_0^n$.  In the present context this may
fail, so in Section~\ref{sec:pw} we define an adapted metric which we use to define the
dynamical Bowen balls instrumental in the estimation of the entropy of $\mu_*$ in Section~\ref{sec:max}.

The structure of the paper is as follows.
We begin by defining in Definition~\ref{def:h_*} the exponential rate of growth in complexity, $h_*$, which counts the number of domains of continuity $\cM_0^n$ of $T^n$.
This quantity dominates the measure-theoretic entropies
of the invariant measures (Theorem~\ref{thm:initial}).  We then proceed to study
the action of a weighted transfer operator, defined in Section~\ref{sec:transfer}.
The Banach spaces we use are similar to those defined in \cite{demers liv} (not \cite{max})
for this
class of maps, yet the operator has significant differences from the transfer operator with
respect to the SRB measure studied in \cite{demers liv}.  By proving a series of growth
and fragmentation lemmas in Sections~\ref{sec:growth} and \ref{sec:lower} that control the
prevalence of short and long connected components of $T^{-n}W$ for local stable manifolds $W$, we are able to establish that the operator has a spectral gap in Section~\ref{sec:spec}.  Finally, in Section~\ref{sec:max},
we construct a measure $\mu_*$ out of the left and right eigenvectors of the transfer operator
and show that it has exponential decay of correlations and that it is
the unique invariant measure with entropy equal to $h_*$.  The properties of the measure $\mu_*$
are summarized in Theorem~\ref{thm:mu}.  
In Corollary~\ref{cor:per}, we derive our asymptotic bound on the growth rate of periodic orbits,
applying results of \cite{lima} and \cite{Bu}. 
Finally, as a byproduct of this approach, uniform
growth rates are established for $\# \cM_0^n$ and the length $|T^{-n}W|$ of stable
manifolds $W$; these are stated in Proposition~\ref{prop:M0n} and
Corollary~\ref{cor:uniform growth}, respectively.


\section{Setting, Definitions and Results}
\label{setting}

In this section, we introduce a set of formal assumptions on our class of piecewise hyperbolic
maps and state the principal results of the paper.


\subsection{Piecewise Hyperbolic Maps}
\label{sec:pw}

Let $M$ be a compact two-dimensional Riemannian manifold, possibly with boundary 
and not necessarily connected, and let 
$T:M \circlearrowleft$ be a piecewise uniformly hyperbolic map in the
sense described below.
There exist a finite number of pairwise disjoint open, simply connected 
regions $\{ M^+_i \}_{i=1}^d$ such that
$\cup_i \overline{M^+_i} = M$ and $\partial M^+_i$ comprises finitely many $\cC^1$
curves of finite length.  We will refer to $\cS^+ = M \setminus \cup_i M_i^+$ as the singularity
set for $T$.

Define $M_i^- = T(M^+_i)$.  We assume that $\cup_i \overline{M^-_i} = M$
and refer to the set $\cS^- = M \backslash \cup_i M_i^-$ as
the singularity set for $T^{-1}$.
We require that $T\in\operatorname{Diff}^2(M \backslash \cS^+,M
\backslash \cS^-)$ and that on each $M^+_i$, $T$ has a $\cC^2$ extension\footnote{This implies in
particular that $\| DT \|$ is bounded on each $M_i^+$, so that this class of maps does not include
dispersing billiards.}
to $\overline{M_i^+}$.  Since the extension of $T$ is defined on $\partial M_i^+$, we will write
$T(\cS^+)$ to denote the set of images of these boundary curves (on which the extension of
$T$ may be multi-valued). 
In this notation, $T(\cS^+) = \cS^-$ and $T^{-1}(\cS^-) = \cS^+$.

On each $M_i$, $T$ is uniformly  hyperbolic: i.e., there exist constants $\Lambda>1$, $\kappa \in (0,1)$
and two $DT$-strictly-invariant 
families of cones $C^u$ and $C^s$, continuous in each 
$\overline{M_i^+}$ which satisfy,
$DT(x) C^u(x) \subsetneq C^u(Tx)$, $DT^{-1}(x) C^s(x) \subsetneq C^s(T^{-1}x)$, and
\begin{equation}
\label{eq:exp def}
\begin{split}
&  \inf_{x\in M\backslash \cS^+} \; \inf_{v \in C^u}
\frac{\| DTv \|}{\|v\|} \; \ge \Lambda \, , \quad
 \inf_{x\in M\backslash \cS^-} \; \inf_{v \in C^s}
\frac{\| DT^{-1}v \|}{\|v\|} \;  \ge \; \Lambda \, , \\ 
&  \qquad \qquad \mbox{and} \quad \kappa \; := \; \inf_{x\in M\backslash \cS^+} \; \inf_{v \in C^s}
\frac{\| DTv \|}{\|v\|}   \, .
\end{split}
\end{equation}
The strict 
invariance of the cone field together with the smoothness properties of the map 
implies that the stable and unstable directions are well-defined for each point
whose trajectory does not meet a singularity line.

In Section~\ref{sec:admissible}, we define narrower cones with the same names
and refer to them as the stable and unstable cones of $T$  
respectively.  We assume the following uniform transversality properties:
there is a uniform positive lower bound on the angle between vectors in $C^s(x)$ and $C^u(x)$
for all $x \in M \setminus \cS^+$,
the tangent vectors to the
singularity curves in $\cS^-$ are bounded away from $C^s$, and those of $\cS^+$
are bounded away from $C^u$; lastly, curves in $\cS^-$ either coincide with, or are uniformly transverse to, curves in $\cS^+$.
As mentioned in the introduction, this class of maps is similar to that studied in
\cite{Pesin1, Li1, Y98, demers liv, baladi gouezel2}; see also \cite{LiWo} for the symplectic case.

\begin{convention}(Doubling boundary points.)
\label{con:pointwise}
It will be convenient in what follows to have $T$ defined pointwise on $M$, but a priori
it is defined only on $\cup_i M_i^+$.  
Since $T$ is $C^2$ up to the closure of each $M_i^+$, we may extend
$T$ to be defined on $\partial M_i^+$, making $T$ multivalued where these boundaries overlap.
Following \cite{Li1}, we 
adopt the convention that the image of such a subset of $M$ under $T$ contains all such points,
and continue to call this extended space $M$.
\end{convention}

We remark that although this convention is made for convenience, it follows from
Theorem~\ref{thm:mu}(a) that the measure $\mu_*$ is independent of how $T$ is defined
on $\partial M_i^+$.

Let $d(\cdot, \cdot)$ denote the Riemannian metric on $M$.  
The following related metric is better adapted to the dynamics.  Define
\begin{equation}
\label{eq:bar d}
\bar d(x,y) = d(x,y), \quad \mbox{whenever $x,y$ belong to the same component $\overline M_i^+$}, 
\end{equation}
and $\bar d(x,y) = 10\, \diam(M)$ otherwise.  Since we have doubled boundary points in $M$
according to Convention~\ref{con:pointwise}, the extended space $M$ is compact in the metric
$\bar d$.

Denote by $\cS_n^+ = \cup_{i=0}^{n-1} T^{-i} \cS^+$ the set of singularity curves for $T^n$ and by 
$\cS_n^- = \cup_{i=0}^{n-1} T^i \cS^-$ the 
set of singularity 
curves for $T^{-n}$.
Let $K(n)$ denote the maximum number of singularity curves
in $\cS_n^-$ or in $\cS_n^+$ which intersect at a single point.  We make the following
assumption
regarding the complexity of $T$.

\bigskip
\noindent
\parbox{.1 \textwidth}{(P1)}
\parbox[t]{.85 \textwidth}{There exist $\alpha_0 > 0$ and an integer
$n_0>0$, such that $\Lambda \kappa^{\alpha_0} > 1$  
and $(\Lambda \kappa^{\alpha_0})^{n_0} > K(n_0)$.    } 

\bigskip
\noindent
Condition (P1) can always be satisfied if $K(n)$ has polynomial growth (as is the case
with a Sinai billiard on a torus); however, since (P1) is required only for some fixed $n_0$, it is
not necessary to control $K(n)$ for all $n$ in order to verify the condition.

\begin{remark}
\label{rem:iterate p1}
If property (P1) holds for $\alpha_0$, then it holds for all
$0 < \alpha < \alpha_0$ with the same $n_0$. 
Notice also that $K(kn_0) \leq K(n_0)^k$ which implies that the
inequality in (P1) can be iterated to make  
$(\Lambda \kappa^{\alpha_0})^{-kn_0} K(kn_0)$
arbitrarily small once (P1) is satisfied for some $n_0$.
\end{remark}

In Section~\ref{sec:admissible} we will define a set of admissible stable curves
$\hW^s$, with tangent vectors belonging to the stable cone, which  we will use to define our
norms.  For $W \in \hW^s$, let $K_n$ denote the number of
smooth connected components of $T^{-n}W$.  For a fixed $N$, by
shrinking the maximum length $\delta_0$ of leaves in $\hW^s$, we can
require that $K_N \leq K(N)+1$. This implies that choosing $N=kn_0$,
we can make $(\Lambda \kappa^{\alpha_0})^{-N}K_N$ arbitrarily small. 

\begin{convention}
\label{convention: n_0=1}
In what follows, we will assume that $n_0 = 1$.  If this is not the
case, we may always consider a higher  iterate of $T$ for which this
is so by assumption (P1).  We then choose $\delta_0$ small 
enough that $K_1 \Lambda^{-1} \kappa^{-\alpha_0}=:\rho<1$.
\end{convention}

We also assume the following.

\bigskip
\noindent
\parbox{.1 \textwidth}{(P2)}
\parbox[t]{.85 \textwidth}{$T$ is topologically mixing and preserves a unique smooth
invariant measure $\musrb$, i.e. there exists $\fsrb \in \cC^1(M_i^+)$ for each $i$ such that
$d\musrb = \fsrb dm$, where $m$ denotes the Riemannian volume on $M$.   } 

\bigskip

\begin{remark}
Property (P1) is standard for piecewise hyperbolic maps, and a variant of it has been used in
\cite{Pesin1, Li1, Y98, demers liv, baladi gouezel2}.  The most common form is only to require the complexity bound
in one direction, for example on $\cS_n^-$ in \cite{Li1, demers liv}.  Here, we assume the symmetric
version on both $\cS_n^-$ and $\cS_n^+$ in order to prove the super-multiplicativity property
for $\# \cM_0^n$, Proposition~\ref{prop:super}.  In fact, the requirement for $\cS_n^+$ 
is used only in the proof of
Lemma~\ref{lem:long elements}.

It follows from the piecewise hyperbolicity of $T$ and (P1) that $T$ admits
an SRB measure \cite[Theorem~1]{Pesin1}.  The requirement that $\musrb$ be smooth in Property (P2) is less essential to our argument.
We use $\musrb$ as our reference measure rather than the Riemannian volume $m$
in order to simplify the estimates involving the transfer operator.  Assuming that $\musrb$
is smooth allows us to prove the embedding lemma, Lemma~\ref{lem:embed}, connecting our Banach spaces
to the standard spaces of distributions.
\end{remark}

Our assumptions on the hyperbolicity of $T$ imply the following uniform expansion and
bounded distortion properties along stable curves, which we record for future use.  There exists $C_e>0$
such that for any $W \in \hW^s$ and $n \ge 0$,
\begin{equation}
\label{eq:one grow}
|T^{-n}W| \ge C_e \Lambda^n |W| \, ,
\end{equation}
where $|W|$ denotes the arc length of $W$ in the metric induced by the Riemannian metric on $M$.

Suppose $W \in \hW^s$ is such that $T^n$ is smooth on $W$ and $T^iW \in \hW^s$,
for $i = 0, \ldots, n$.  We denote by $J_WT^n$ the Jacobian of $T^n$ along $W$ with respect to
arc length.
There exists $C_d>0$, independent of $W$, such that for all
$x,y \in W$ and all $n \ge 0$,
\begin{equation}
\label{eq:distortion}
\left| \frac{J_WT^n(x)}{J_WT^n(y)} - 1 \right| \le C_d d_W(x,y) \, ,
\end{equation}
where $d_W(\cdot, \cdot)$ denotes arc length distance along $W$.


\subsection{A Definition of Topological Entropy}
\label{sec:ent def}

Following \cite{max}, for $k, n \ge 0$, let $\cM_{-k}^n$ denote the set of maximal connected components of 
$M \setminus (\cS^+_n \cup \cS^-_k)$, where we define $\cS^{\pm}_0 = \emptyset$.  Note that by definition, 
elements of $\cM_{-k}^n$ are open in $M$.  With this notation, $\cM_0^n$ denotes the set of
maximal, open, simply connected components of $M$ on which $T^n$ is continuous, while
$\cM_{-n}^0$ has the analogous property for $T^{-n}$.  We remark also that the requirement
that each $M_i^+$ be open and simply connected prevents the partition $\cM_0^1$ from
being trivial, and implies in particular that the diameter of elements of $\cM_{-n}^n$ tends
to 0 as $n$ gets large.\footnote{Thus, if one wants to apply the present results to a smooth
map, for example a toral automorphism, one should first partition the torus into a finite number of 
simply connected `rectangles' with boundaries transverse to $C^u$ and $C^s$.  Then
Theorem~\ref{thm:mu} implies that the rate of growth in cardinality of dynamical refinements
of this partition, $h_*$,
will equal the topological entropy of the automorphism.}
Similarly, the transversality assumptions coupled with the finiteness requirement on the
number of smooth curves in $\cS^+$ guarantee that $\# \cM_{-k}^n$ is finite for each
$k$ and $n$.

\begin{defin}(Topological entropy of $T$.)
\label{def:h_*}
Define $\ds h_*(T) = \limsup_{n \to \infty} \frac 1n \log \left( \# \cM_0^n \right)$.
\end{defin}

By definition, if $A \in \cM_0^n$, then $T^nA \in \cM_{-n}^0$, so that
$\# \cM_0^n = \# \cM_{-n}^0$.  Thus $h_*(T) = h_*(T^{-1})$, i.e. this definition is
symmetric in time. Indeed, the limsup in the definition is in fact a limit, which follows
from Proposition~\ref{prop:M0n}.

We begin by establishing that the quantity $h_*$ is finite.

\begin{lemma}
\label{lem:finite}
For a piecewise hyperbolic map $T$ as defined in Section~\ref{sec:pw}, 
but not necessarily satisfying conditions (P1) and (P2), the quantity $h_* < \infty$.
\end{lemma}

\begin{proof}
The elements of $\cM_0^1$ are simply the domains $M_i^+$.  For any $n \ge 1$, elements of 
$\cM_0^{n+1}$ are created by (the image under $T^{-1}$ of) the connected components of the intersection of 
an element of $\cM_0^n$ with one
of the domains $M_i^-$.  
By assumption, $\cS^+$ and $\cS^-$ comprise finitely many $\cC^1$ curves which either coincide
or are uniformly transverse.  Since $T$ is $\cC^2$ on the closure of each $M_i^+$, 
the same is true of the sets $\cS_n^+$ and $\cS^-$.   Moreover, elements of $\cS_n^+$ 
have a uniform bound (in $n$) on their derivative.

Consider the intersection $A \cap M_i^-$ for $A \in \cM_0^n$.
Connected components of this set are created by intersections of
$\partial A$ with elements of $\cS^-$.  Since $\partial A \subset \cS^+_n$, 
by the compactness of $M$ and uniform transversality, $\partial A$ can intersect each smooth
curve in $\cS^-$ a finite number of times, with uniform
upper bound $B>0$ independent of $n$.  Thus the number of connected
components of $A \cap M_i^-$ is bounded by
$B (\# \cS^-)$.
Since this bound holds for each $A \in \cM_0^n$, we have
\[
\# \cM_0^{n+1} \le (\# \cM_0^n) B (\# \cS^-) \le d B^n (\# \cS^-)^n \, ,
\]
where $d$ is the number of domains $M_i^+$.
\end{proof}

In order to connect $h_* = h_*(T)$ to the dynamical refinements of a fixed partition, for
each $k \in \mathbb{N}$, define 
$\cP_k$ to be the maximal connected components of $M$ on which $T^k$ and $T^{-k}$
are continuous. That is, $\cP_k$ is the partition of $M$ defined by
$M \setminus (\cS_k^+ \cup \cS_k^-)$, together with the boundary curves associated to each element,
according to Convention~\ref{con:pointwise}.
If we let $\hP_k$ denote the collection of interiors of elements of $\cP_k$, then we have
$\hP_k = \cM_{-k}^k$.

For $n \ge 1$, 
define $\cP_k^n = \bigvee_{i=0}^n T^{-i}\cP_k$.  $\cP_k^n$ is still a pointwise partition of $M$,
yet its elements may not be open sets, and it may occur that $\cP_k^n$ contains isolated points
due to multiple boundary curves intersecting at one point.  Furthermore, we do not
assume that the elements of $\cP_k^n$ are connected sets.\footnote{Contrast this with
\cite[Lemma~3.1]{max}, where the analogous construction yields connected elements
due to the property of continuation of singularities enjoyed by dispersing billiards.}
Thus, although the collection of interiors $\hP_k^n$ is a partition of $M \setminus (\cS_{k+n}^+ \cup \cS_k^-)$, it may be that $\hP_k^n \neq \cM_{-k}^{k+n}$.

Our next lemma provides a rough upper bound on the number of isolated points that can be
created by refinements of $\cP_k$.  Let $\# \cS^{\pm}$ denote the number of smooth components
of $\cS^{\pm}$.

\begin{lemma}
\label{lem:isolated}
For each $k, n \ge 1$, the number of isolated points in $\cP_k^n$ is at most
\[
2 (\# \cS^- + \# \cS^+) \sum_{j=1}^{k+n} \# \cM_0^j.
\]
\end{lemma}

\begin{proof}
By Convention~\ref{con:pointwise}, there are no isolated points in $\cP_1$.
Next, for each $n \ge1$, 
at time $n$, isolated points in $\cP_1^n$ can be produced by intersections of
corner points in the boundary of $\cP_1^{n-1}$ with elements of $\cS^-$.
Moreover, each pair of smooth curves $S \in \cS_n^+$ and $S' \in \cS^-$ intersect
at most twice per element of $\cM_0^n$.  Thus the number of new isolated points
created at time $n$ is at most $2 \# \cS^- \# \cM_0^n$.  Applying this estimate inductively,
we have
\[
\mbox{number of isolated points in $\cP_1^n$}  \le 2 \# \cS^- \sum_{j=1}^n \# \cM_0^j \, .
\]
Next, for each $k$, applying a similar inductive argument to $T^{-1}$, we have
\[
\begin{split}
\mbox{number of isolated points in $\cP_k$}&  \le 2 \# \cS^+ \sum_{j=1}^k \# \cM_{-j}^0 + 2 \# \cS^- \sum_{j=1}^k \# \cM_0^j  \\
& \le 2(\# \cS^+ + \# \cS^-) \sum_{j=1}^k \# \cM_0^j \, ,
\end{split}
\]
where we have used the fact that $\# \cM_0^j = \# \cM_{-j}^0$.  Finally, refining $\cP_k$,
we create at most $2 \# \cS^- \# \cM_0^{k+j}$ new isolated points in $\cP_k^j$ at time $j$.
Summing over $j \le n$, we complete the proof of the lemma.
\end{proof}


\subsection{Statement of Main Results}
\label{sec:main}

Our first result establishes a connection between the rates of growth of $\# \cP_k^n$ and
$\# \cM_0^n$, and uses this to prove that $h_*$ dominates the measure-theoretic
entropies of the invariant measures of $T$.

\begin{theorem}
\label{thm:initial}
Let $T$ be a piecewise hyperbolic map as defined in Section~\ref{sec:pw}, 
but not necessarily satisfying conditions (P1) and (P2).
\begin{itemize}
\item[a)] For each $k,n \ge 1$, $\# \hP_k^n \le  \#\cM_{-k}^{k+n}$ and $\# \cP_k^n \le C (k+n) \# \cM_{-k}^{k+n}$,
for some $C>0$ depending only on $T$. 
\item[b)] For all $k \ge 1$, $\ds \limsup_{n \to \infty} \frac 1n \log (\# \cM_{-k}^n) = h_*$.
\item[c)]
$\ds \sup_k \lim_{n \to \infty} \frac 1n \log \# \cP_k^n = \sup_k \lim_{n \to \infty} \frac 1n \log \# \hP_k^n\le h_*$.
\item[d)]
$\ds h_* \ge \sup \{ h_\mu(T) : \mu \mbox{ is an invariant probability measure for $T$} \} .$
\end{itemize}
\end{theorem}

\begin{proof}

a) The first inequality is straightforward since by definition, both $\hP_k^n$ and $\cM_{-k}^{k+n}$ are
partitions of $M \setminus (\cS^+_{n+k} \cup \cS^-_k)$, yet $\hP_k^n$ may have disconnected
components.  Thus $\cM_{-k}^{k+n}$ is a refinement of $\hP_k^n$.  The second inequality follows
by noting that $\# \cP_k^n$ equals $\# \hP_k^n$ plus isolated points, and then applying 
Lemma~\ref{lem:isolated}.

\medskip
\noindent
b) The value of the limsup is the same for each $k$ since by definition, 
$A \in \cM_{-k}^n$ if and only if $T^kA \in \cM_0^{n+k}$.  Thus $\# \cM_{-k}^n = \# \cM_0^{n+k}$.

\medskip
\noindent
c) We first remark that $\# \cP_k^{n+m} \le \# \cP_k^n \# \cP_k^m$, and also
$\# \hP_k^{n+m} \le \# \hP_k^n \# \hP_k^m$ (which can be proved as in 
\cite[Lemma~3.3]{max}), thus the two limits in part (c) exist by subadditivity.  
The fact that both limits are bounded by $h_*$ follows from
parts (a) and (b) of the theorem.

\medskip
\noindent
d)  Let $\mu$ denote a $T$-invariant probability measure.
The assumptions of uniform hyperbolicity imply that
both $T$ and $T^{-1}$ are expansive with respect to the metric
$\bar d$ defined in \eqref{eq:bar d}: 
\begin{equation}
\label{eq:expansive}
\mbox{There exists $\ve_0 >0$ such that if $\bar d(T^jx, T^jy) < \ve_0$ for all $j \in \mathbb{Z}$,
then $x=y$.}
\end{equation}
By \eqref{eq:exp def}, the uniform transversality of stable and unstable cones, and the
assumption that each $M_i^+$ is simply connected,
the maximum diameter of elements of $\cM_{-k}^k$ (and hence of $\cP_k$) is bounded
by $C \Lambda^{-k}$.  Choosing $k$ large enough that 
$C \Lambda^{-k} \le \ve_0$, we conclude that $\cP_k$ is a generator for
$T$ \cite[Theorem~5.23]{walters}.
Then by \cite[Theorem~4.22]{walters}, 
\[
h_\mu(T) = h_\mu(T, \cP_k) = \lim_{n \to \infty} \frac 1n H_\mu(\cP_k^n) 
\le \lim_{n \to \infty} \frac 1n \log (\# \cP_k^n) \le h_* \, ,
\]
applying part (c) of the present theorem.  Thus $h_\mu(T) \le h_*$.
\end{proof}

Next we state our main theorem, which requires the additional hypotheses (P1) and (P2).

\begin{theorem}
\label{thm:mu}
Let $T$ be a piecewise hyperbolic map as defined in Section~\ref{sec:pw}, 
satisfying conditions (P1) and (P2).

There exists a $T$-invariant probability measure $\mu_*$ 
with the following properties.
\begin{itemize}
  \item[a)] The measure $\mu_*$ has no atoms, and there exists $C>0$ such that for any $\ve>0$,
\[
\mu_*(\cN_\ve(\cS^{\pm})) \le C \ve^{1/p} \, ,
\]
where $p>1$ is from \eqref{eq:restrict} and $\cN_\ve(\cdot)$ denotes the $\ve$-neighborhood
of a set in the Riemannian metric on $M$.  This implies in particular, that $\mu_*$-a.e.
$x \in M$ has a stable and unstable manifold of positive length, and that $x$ approaches
$\cS^{\pm}$ at a subexponential rate.
  \item[b)] $\mu_*(O) > 0$ for any open set $O \subset M$.
  \item[c)] $(T^n, \mu_*)$ is ergodic for all $n \in \mathbb{Z}^+$.
  \item[d)] $\mu_*$ has exponential
decay of correlations against H\"older continuous functions.  
  \item[e)] The measure
$\mu_*$ is the unique $T$-invariant probability measure satisfying
$h_{\mu_*}(T) = h_*$.  
\end{itemize}
\end{theorem}

Theorem~\ref{thm:mu} will be proved in Section~\ref{sec:max}.  In particular, items
(a)-(c) are proved in Section~\ref{sec:hyper}, item (d) is proved in Proposition~\ref{prop:decay},
and item (e) is proved in Sections~\ref{sec:entropy} and \ref{sec:unique}.

\begin{cor}
Let $T$ be a piecewise hyperbolic map as defined in Section~\ref{sec:pw}, 
satisfying conditions (P1) and (P2).

$T$ satisfies the following variational principle: For all $k \ge 0$,
\[
\lim_{n \to \infty} \frac 1n \log \left( \# \cM_{-k}^n \right) = h_* = \sup \{ h_\mu(T) : \mu \mbox{ is an invariant probability measure for $T$} \} . 
\]
\end{cor}

\begin{proof}
The fact that the limit defining $h_*$ exists (rather than simply the $\limsup$ from
Definition~\ref{def:h_*}) follows from Proposition~\ref{prop:M0n}, and the independence
from $k$ follows from Theorem~\ref{thm:initial}(b).
The second equality follows from Theorem~\ref{thm:initial}(d) together with 
Theorem~\ref{thm:mu}(e).
\end{proof}

Theorem~\ref{thm:mu}(a) implies that $\int_M |\log d(x, \cS^{\pm})| \, d\mu_* < \infty$
(see Corollary~\ref{cor:atomic}(c)), so that $\mu_*$ is $T$-adapted in the language
of \cite{lima}.  This allows us to make the following connection to the growth of
periodic orbits of $T$.  Let $P_n(T) = \{ x \in M : \# \{T^kx : k \in \mathbb{Z} \} = n \}$ 
denote the set of
points of prime period $n$ for $T$.

\begin{cor}
\label{cor:per}
Under the assumptions of Theorem~\ref{thm:mu},
$\displaystyle \liminf_{n \to \infty} \# P_n(T) e^{-n h_*} = 1$. 
\end{cor}

\begin{proof}
The proof relies on the construction of a countable Markov partition for 
hyperbolic maps with singularities carried out in \cite{lima}.  The class of maps
in the present paper satisfy conditions (A1)-(A6) in \cite{lima}, which are general
enough to admit dispersing billiards.
Since $\mu_*$ is $T$-adapted and hyperbolic (see Corollary~\ref{cor:atomic}),
we may apply \cite[Corollary~1.2]{lima} to conclude that there exist $p \ge 1$ and $C>0$ such that
the number of points of period $np$ for $T$ is at least $C e^{np h_*}$ for all $n$
sufficiently large.

Next, applying \cite[Main~Theorem]{Bu} as in \cite[Theorem~1.5]{Bu}, we conclude that
we may take $p=1$ and asymptotically, $C=1$ for large $n$.
\end{proof}

In the course of proving the growth lemmas in Section~\ref{sec:banach},
we establish the following uniform bounds on the growth of $\# \cM_0^n$, which may
be of independent interest, and are needed for the proof of uniqueness in
Section~\ref{sec:unique}. 

\begin{proposition}
\label{prop:M0n}
There exists a constant $C_\#>0$ such that for all $n \ge 1$,
\[
C_\# e^{n h_*} \le \# \cM_0^n = \# \cM_{-n}^0 \le C_\#^{-1} e^{n h_*} \, .
\]
\end{proposition}

\begin{proof}
The upper bound is Corollary~\ref{cor:upper M}, while the lower bound is
Lemma~\ref{lem:lower M}.
\end{proof}

\begin{cor}
\label{cor:uniform growth}
There exists $\bar C>0$ such that for all stable curves $W \in \hW^s$ with $|W| \ge \delta_1/3$ and all $n \ge n_1$, where both $\delta_1>0$ and $n_1$ are from
\eqref{eq:delta_1}, we have
\[
\bar C e^{n h_*} \le |T^{-n}W| \le \bar C^{-1} e^{n h_*} \, .
\]
\end{cor}

\begin{proof}
Let $W \in \hW^s$ with $|W| \ge \delta_1/3$.  We use the notation of Section~\ref{sec:growth} regarding the connected components $\cG_n(W)$ of $T^{-n}W$.
Lemma~\ref{lem:growth}(b), Lemma~\ref{lem:lower}
and Proposition~\ref{prop:M0n} together yield,
\[
c_0 C_\# e^{n h_*} \le c_0 \# \cM_0^n \le    \# \cG_n(W) \le C \delta_0^{-1} \# \cM_0^n \le C \delta_0^{-1} C_\#^{-1} e^{n h_*} \, .
\]
Then on the one hand, 
\[
|T^{-n}W| = \sum_{W_i \in \cG_n(W)} |W_i| \le \delta_0 \# \cG_n(W) \, ,
\]
since each element of $\cG_n(W)$ has length at most $\delta_0$,
completing the upper bound of the corollary.  On the other hand, by \eqref{eq:lower spec},
\[
|T^{-n}W| = \sum_{W_i \in \cG_n^{\delta_1}(W)} |W_i| \ge \tfrac{2\delta_1}{9} \# \cG_n(W) \, ,
\]
proving the lower bound.
\end{proof}


\section{Banach Spaces and Growth Lemmas}
\label{sec:banach}

In this section we define the Banach spaces we will use in the analysis of the transfer operator
and prove several key lemmas controlling the growth in complexity of $T^n$.


\subsection{Stable Curves}
\label{sec:admissible}

We begin with a definition of stable curves as graphs of functions in local charts,
following \cite{demers liv}.   We will use the fact that the uniform hyperbolicity
of $T$ guarantees the existence of stable $E^s(x)$ and unstable $E^u(x)$
directions in the tangent space $\cT_xM$ at Lebesgue-almost-every $x \in M$.

For $\tau$ sufficiently small, we define the stable cone at $x \in M$ by
\[
\hat C^s(x) = \{ u + v \in \cT_xM : u \in E^s(x), v \perp E^s(x), \| v \| \le \tau \| u \| \} \, .
\]
Define $\hat C^u(x)$ analogously.  These families of cones are strictly invariant,
$DT^{-1}(x) \hat C^s(x) \subsetneq \hat C^s(T^{-1}x)$ and $DT(x) \hat C^u(x) \subsetneq C^u(Tx)$.

For each $i$, we choose a finite number of coordinate charts
$\{ \chi_j \}_{j=1}^L$, whose domains $R_j$ are either $(-r_j, r_j)^2$ if $\chi_j$ maps only to the
interior of $M_i^+$, or $(-r_j, r_j)$ restricted to one side of a piecewise $\cC^1$ curve
(the preimage of a piece of $\partial M_i^+$) which we place so that it passes through the origin.
For each $j$, $R_j$ has a centroid $x_j$, and $\chi_j$ satisfies,
\begin{itemize}
  \item[(a)]  $D\chi_j(x_j)$ is an isometry;
  \item[(b)]  $D\chi_j(x_j) \cdot (\mathbb{R} \times 0) = E^s(\chi_j(x_j)$;
  \item[(c)]  The $\cC^2$-norm of $\chi_j$ and its inverse are bounded by $1+\tau$;
  \item[(d)]  There exists $c_j \in (\tau, 2\tau)$ such that the cone $C_j = \{ u+v \in \mathbb{R}^2 :
  u \in \mathbb{R} \times \{ 0 \}, v \in \{ 0 \} \times \mathbb{R}, \| v \| \le c_j \| u \| \}$ satisfies:
  For each $y \in R_j$ such that $\chi_j(y) \notin \cS^-$, $D\chi_j(y) C_j \supset \hat C^s(\chi_j(y))$,
  and $DT^{-1}(D\chi_j(y)C_j) \subset \hat C^s(T^{-1}(\chi_j(y)))$;
  \item[(e)]  $M_i^+ \subset \cup_{j=1}^L \chi_j(R_j \cap (-\frac{r_j}{2}, \frac{r_j}{2})^2)$.
\end{itemize}

Choose $r_0 \le \frac 12 \min_j r_j$; $r_0$ may be further reduced later, depending on $\delta$.  
Fix $B < \infty$
and consider the set of functions
\[
\Xi := \{ F \in \cC^2([-r,r], \mathbb{R}) : r \in (0, r_0], F(0)=0, |F|_{\cC^1} \le \tau, |F|_{\cC^2} \le B \} \, .
\]
Define $I_r = (-r,r)$.  For $x \in R_j \cap (-r_j/2, r_j/2)^2$ such that $x + (t, F(t)) \in R_j$ for 
$t \in I_r$, define $G(x,r,F)(t) := \chi_j(x + (t, F(t))$ for $t \in I_r$, i.e. $G(x,r,F)$ is the lift of the graph of $F$ to $M$.  To abbreviate notation, we will refer to $G(x,r,F)$ as $G_F$.
It follows from the construction that $|G_F|_{\cC^1} \le (1+\tau)^2$ and $G_F^{-1} \le 1 + \tau$.

Our set of admissible stable curves is defined by,
\[
\hW^s := \{ W = G(x,r,F)(I_r) : x \in R_j \cap (r_j/2, r_j/2)^2, r \le r_0, F \in \Xi \} \, .
\]
If necessary, we reduce $r_0$ so that $\sup_{W \in \hW^s} |W| \le \delta_0$, where $\delta_0$
is the length scale chosen in Convention~\ref{convention: n_0=1}.
Due to the uniform hyperbolicity of $T$, if $T^{-n}\hW^s$ represents the connected components of $T^{-n}W$ for
$W \in \hW^s$, then choosing $B$ large enough, it follows that $T^{-n}\hW^s \subset \hW^s$,
up to subdivision of long curves. With this choice of $B$, 
the set of real local stable manifolds of
length at most $\delta_0$, which we denote by $\cW^s$, satisfies $\cW^s \subset \hW^s$.

Next, we define two notions of distance\footnote{Neither of these distances will satisfy the
triangle inequality, but that is irrelevant for our purposes.} which are used in the definition of our norms, namely
the strong unstable norm.  For two curves $W_1(\chi_{i_1}, x_1, r_1, F_1)$ and 
$W_2(\chi_{i_2}, x_2, r_2, F_2)$, we define the distance between them to be,
\[
d_{\cW^s}(W_1, W_2) = \eta(i_1, i_2) + |x_1-x_2| + |r_1-r_2| + |F_1 - F_2|_{\cC^1(I_{r_1} \cap I_{r_2})},
\]
where $\eta(i_1,i_2) = 0$ if $i_1=i_2$ and $\eta(i_1, i_2) = \infty$ otherwise, i.e. we only compare
curves in the same chart.

Given $W_1, W_2$ with $d_{\cW^s}(W_1, W_2) < \infty$ and two functions 
$\psi_i \in \cC^0(W_i)$, we define the distance between them to be
\[
d_0(\psi_1, \psi_2) = |\psi_1 \circ G_{F_1} - \psi_2 \circ G_{F_2} |_{\cC^0(I_{r_1} \cap I_{r_2})} \, .
\]


\subsection{Transfer operator}
\label{sec:transfer}

The main tool we will use to construct the measure of maximal entropy is a weighted
transfer operator, $\cL$.  Because we do not have a conformal measure at our
disposal a priori, we will define the transfer operator acting on distributions defined
via local stable manifolds.  
Let $\widetilde \cW^s$ denote the set of maximal connected local
stable manifolds of $T$ restricted to each $M_i^+$.  
Note that such manifolds have uniformly bounded length due to the
the finite diameter of $M$ and the assumption that $M_i^+$ is simply connected.  
Due to the uniform hyperbolicity
of $T$, $\musrb$-almost every point in $M$ has a stable manifold of positive length.


For any local stable manifold $W$, and $\alpha \in (0,1]$, define the $\alpha$-H\"older norm of a test function 
$\psi : M \to \mathbb{C}$ by
\begin{equation}
\label{eq:holder def}
| \psi |_{\cC^\alpha(W)} = | \psi |_{\cC^0(W)} + H_W^\alpha(\psi)
:= \sup_W |\psi| + \sup_{x \neq y \in W} \frac{|\psi(x) - \psi(y)|}{d_W(x,y)^\alpha} \, ,
\end{equation}
where $d_W(\cdot, \cdot)$ denotes distance induced by the Riemannian metric
restricted to $W$.
Let $\tilde \cC^\alpha(W)$ denote the set of functions in $\cC^0(W)$ with finite
$| \cdot |_{\cC^\alpha(W)}$ norm.  With this notation, $\tilde \cC^1(W)$ denotes the
set of Lipschitz functions on $W$.

Analogously, for each $n \ge 0$, define 
$H^\alpha_{\widetilde \cW^s}(\psi) = \sup_{W \in \widetilde \cW^s} H^\alpha_W(\psi)$,
and
\[
\tilde \cC^\alpha(\widetilde \cW^s) = \{ \psi : M \to \mathbb{C} \mid |\psi|_\infty + H^\alpha_{\widetilde \cW^s}(\psi) < \infty \} \, .
\] 
The set $\tilde \cC^\alpha(\widetilde \cW^s)$ together with the norm 
$|\psi|_{\cC^\alpha(\widetilde \cW^s)} := |\psi|_\infty + H^\alpha_{\widetilde \cW^s}(\psi)$ is a Banach space.

Since stable manifolds cannot be cut under $T^n$, if $W \in \widetilde \cW^s$, then
$T^n W \subset V \in \widetilde \cW^s$ for each $n \ge 0$.
This together with the uniform hyperbolicity of $T$ and \eqref{eq:exp def} implies that if
$\psi \in \cC^\alpha(\widetilde \cW^s)$, then $\psi \circ T \in \cC^\alpha(\widetilde \cW^s)$
(see also \eqref{eq:holder}).

Then if $f \in (\tilde \cC^\alpha(\widetilde \cW^s))^*$ belongs to the dual of 
$\cC^\alpha(\widetilde \cW^s)$, the operator $\cL : (\tilde \cC^\alpha(\widetilde \cW^s))^* \to
(\tilde \cC^\alpha(\widetilde \cW^s))^*$ is defined by,
\begin{equation}
\label{eq:trans def}
\cL f (\psi) = f \left( \frac{\psi \circ T}{J^sT} \right) \quad \forall \psi \in \cC^\alpha(\widetilde \cW^s) \, ,
\end{equation}
where $J^sT$ denotes the stable Jacobian of $T$.  By 
\eqref{eq:distortion}, it follows that\footnote{For $x \in W \in \widetilde \cW^s$,
$J^sT(x) = J_WT(x)$.} $J^sT^n \in \tilde \cC^1(\widetilde \cW^s)$ for each $n \ge 1$.

If $f \in \cC^0(M)$, then we identify $f$ with a signed measure absolutely continuous with
respect to $\musrb$.  We denote this integration by,
\[
f(\psi) = \int_M \psi \, f \, d\musrb \, ,
\]
for $\psi \in \cC^0(M)$.  With this identification, we consider
$\cC^0(M) \subset (\tilde \cC^\alpha(\widetilde \cW^s))^*$.  Then also
by \eqref{eq:trans def}, for any $n \ge 1$,
$\cL^n f$ is absolutely continuous with respect to $\musrb$ with density,
\begin{equation}
\label{eq:trans}
\cL^n f = \frac{f \circ T^{-n}}{J^sT^n \circ T^{-n}} \, .
\end{equation}


\subsection{Definition of Norms}
\label{sec:norms}

Let $\cW^s$ denote those local stable manifolds having length at most $\delta_0$,
where $\delta_0$ is from Convention~\ref{convention: n_0=1}.  Note that
$\cW^s \subset \hW^s$, yet $\cW^s \not\subset \widetilde \cW^s$ since 
$\widetilde \cW^s$ contains only maximal local stable manifolds (which are necessarily disjoint), 
while $\cW^s$ contains stable manifolds of any length less than $\delta_0$, many of which
may overlap.
We will define our norms by integrating on elements of $\cW^s$ against
H\"older continuous test functions.  

For $W \in \cW^s$ and $\alpha >0$, let $\cC^\alpha(W)$ denote the closure of 
$\tilde \cC^1(W)$ in the $\cC^\alpha$ norm, defined in \eqref{eq:holder def}.\footnote{This
space is strictly smaller than the set of $\cC^\alpha$ functions, yet contains $\cC^{\alpha'}$
for each $\alpha' > \alpha$.  We adopt this usage in order that the embedding
of our strong space in our weak space is injective (Lemma~\ref{lem:include}).}
In this notation, then $\cC^1(W) = \tilde \cC^1(W)$.

Now given a function $f \in \cC^1(M)$, define the weak norm of $f$ by
\[
|f|_w = \sup_{W \in \cW^s} \sup_{\substack{\psi \in \cC^1(W) \\ |\psi|_{\cC^1(W)} \le 1}}
\int_W f \, \psi \, dm_W \, ,
\]
where $m_W$ denotes arc length along $W$.  Let $|W| = m_W(W)$.

Next, choose $\alpha, \beta < 1$ and $p > 1$ such that 
\begin{equation}
\label{eq:restrict}
0 < 2\beta \le 1/p \le 1 - \alpha \le \alpha_0,
\qquad
\mbox{and} \quad 1/p < \alpha \, .
\end{equation}
Define the strong stable norm of $f$ by
\[
\| f \|_s = \sup_{W \in \cW^s} \sup_{\substack{\psi \in \cC^\alpha(W) \\ |\psi|_{\cC^\alpha(W)} \le |W|^{-1/p}}} \int_W f \, \psi \, dm_W \, .
\]
Recalling the notion of distance $d_{\cW^s}(\cdot , \cdot)$
between curves $W \in \cW^s$ and the distance $d_0(\cdot, \cdot)$ between test functions on
nearby curves defined in Section~\ref{sec:admissible} and fixing $\ve_0 \le r_0$, we define the strong unstable norm of $f$ by,
\[
\| f \|_u = \sup_{\ve \le \ve_0} \sup_{\substack{W_1, W_2 \in \cW^s \\ d_{\cW^s}(W_1, W_2) \le \ve}} 
\sup_{\substack{|\psi_i|_{\cC^1(W_i)} \le 1 \\ d_0(\psi_1, \psi_2) = 0}} \ve^{-\beta}
\left| \int_{W_1} f \, \psi_1 \, dm_{W_1} - \int_{W_2} f \, \psi \, dm_{W_2} \right| \, . 
\]
Define the strong norm of $f$ by $\| f \|_{\cB} = \| f\|_s + c_u \| f \|_u$, where $c_u >0$
is a constant to be chosen in the proof of Lemma~\ref{lem:radius}.

Finally, our weak space $\cB_w$ is defined to be the completion of $\cC^1(M)$
in the weak norm, $| \cdot |_w$, while our strong space $\cB$ is defined to be the completion
of $\cC^1(M)$ in the strong norm $\| \cdot \|_{\cB}$.

\begin{remark}
\label{rem:why real}
The definition of our spaces $\cB$ and $\cB_w$ is nearly the same as that in 
\cite[Section~2.2]{demers liv}, the key difference being that the norms in 
\cite{demers liv} integrate along cone-stable curves $\hW^s$, while our norms here
integrate on local stable manifolds $\cW^s$.  This change is necessary since the
potential for our weighted transfer operator, $1/J^sT$, is H\"older continuous along
real stable manifolds, yet may only be measurable along arbitrary stable curves.  By
restricting our norms to this smaller set of curves, we are able to prove the
essential Lasota-Yorke inequalities, Proposition~\ref{prop:LY}.
\end{remark}


\subsection{Preliminary facts about the Banach spaces}
\label{sec:prelim}

\begin{lemma} 
\label{lem:piece}
Let $\cQ$ be a (mod 0 w.r.t. $\musrb$) finite partition of $M$ into open, simply connected sets such that
there exist constants $\bar K, C_{\cQ} > 0$ such that for each $Q \in \cQ$, and $W \in \cW^s$,
$Q \cap W$ comprises
at most $\bar K$ connected components and for any $\ve >0$, 
$m_W(\cN_\ve(\partial Q) \cap W) \le C_{\cQ} \ve^{1/2}.$

\begin{itemize}
  \item[a)] Let $\gamma > \beta/(1-\beta)$ and suppose $\vf$ is a function on $M$ such that
$\sup_{Q \in \cQ} |\vf|_{\cC^\gamma(Q)} < \infty$.  Then $\vf \in \cB$.
  \item[b)]  There exists $C>0$ such that if $\vf$ is such that $\sup_{Q \in \cQ} |\vf|_{\cC^1(Q)} < \infty$ and $f \in \cB$,
  then $\vf f \in \cB$ and $\| \vf f \|_{\cB} \le C \| f \|_{\cB} \sup_{Q \in \cQ} |\vf|_{\cC^1(Q)}$.
\end{itemize}
\end{lemma}

\begin{proof}
To prove (a), a function $\vf$ as in the statement of the lemma can be approximated by $\cC^1$ functions
using mollification precisely as in \cite[Lemma~3.5]{demzhang14}.  Part (b) follows along
similar lines using \cite[Lemma~5.3]{demzhang14}.  Both proofs use the restrictions in
\eqref{eq:restrict}
we have assumed for the parameters appearing in the norms.  In particular.
we need $\beta \le 1/(2p)$, rather than simply $\beta \le 1/p$, due to the weak transversality
condition assumed on $\partial \cQ$.
\end{proof}

\begin{lemma}
\label{lem:embed}
Let $f \in \cC^1(M)$ and $\psi \in \tilde \cC^1(\widetilde \cW^s)$.  Then,
\[
|f(\psi)| = \left| \int_M f \, \psi \, d\musrb \right| \le C |f|_w (|\psi|_\infty + H^1_{\widetilde \cW^s}(\psi)) \, .
\]
\end{lemma}

\begin{proof}
Let $f \in \cC^1(M)$ and $\psi \in \tilde \cC^1(\widetilde \cW^s)$.
We will estimate
\[
f(\psi) = \int_M f \, \psi \, d\musrb \, .
\]
To this end, we choose a foliation $\cF = \{ W_\xi \}_{\xi \in \Xi} \subset \cW^s$ of maximal local stable manifolds subdivided according to the length scale $\delta_0$.  We then disintegrate the 
measure $\musrb$ into conditional measures
$\musrb^\xi$ on $W_\xi \in \cF$ and a factor measure $\hatmusrb(\xi)$ on the index set
$\Xi$ of stable manifolds.  Since $\musrb$ is smooth by assumption $(P2)$, it follows
from \cite[Proposition~6]{Pesin1} (see also \cite[eq.~(3.7)]{chernov pw}) 
that the conditional measures $\musrb^\xi$ are absolutely continuous with respect to arc length,
$d\musrb^\xi = |W_\xi|^{-1} g_\xi dm_{W_\xi}$, where $g_\xi$ is given by\footnote{Both \cite{Pesin1} and \cite{chernov pw} give the analogous formula for the conditional measures of 
$\musrb$ on unstable manifolds.  Yet, due to our assumption (P2), $\musrb$ is an SRB measure
for $T^{-1}$ as well, and so enjoys the analogous properties on stable manifolds of $T$.}
\[
\frac{g_\xi(x)}{g_\xi(y)} = \lim_{n \to \infty} \frac{J_{W_\xi}T^n(x)}{J_{W_\xi}T^n(y)}
\quad \mbox{for all } x,y \in W_\xi \, .
\]
This characterization, plus the normalization $\musrb^\xi(W_\xi) =1$, uniquely determines
$g_\xi$.
It follows from a standard estimate\footnote{Note $J_{W_\xi}T^n(z) = \prod_{j=0}^{n-1} J_{T^jW_\xi}T(T^jz)$ and for brevity let $g_n = J_{W_\xi}T^n$.   
The limit of $g_n(x)/g_n(y)$ exists if the limit of $\log (g_n(x)/g_n(y))$ exists.  Now for $n, k \ge 1$, 
we may estimate using \eqref{eq:one grow} and \eqref{eq:distortion}, 
\[
\left| \log \frac{g_n(x)}{g_n(y)} - \log \frac{g_{n+k}(x)}{g_{n+k}(y)} \right|
= \log \frac{J_{T^nW_\xi}T^k(T^nx)}{J_{T^nW_\xi}T^k(T^ny)} 
\le C_d d_{T^nW_\xi}(T^nx, T^ny) \le C_d C_e^{-1} \Lambda^{-n} d_{W_\xi}(x,y) \, ,
\]
so that the sequence $\log (g_n(x)/g_n(y))$ is Cauchy and therefore converges. 
Thus the limit defining $g_\xi$ exists.  A similar estimate shows that
$g_n$ is log-Lipschitz with Lipschitz constant at most $C_d$, bounded independently of $n$, and so this bound carries over to $g_\xi$.}  
and \eqref{eq:distortion}  that $g_\xi$ is uniformly
log-Lipschitz continuous
on $W_\xi$, i.e. there exists $C_g \ge 1$ such that
\begin{equation}
\label{eq:log g}
0 < C_g^{-1} \le \inf_{\xi \in \Xi} \inf_{W_\xi} g_\xi \le  \sup_{\xi \in \Xi} |g_\xi|_{\cC^1(W_\xi)} \le C_g < \infty \, .
\end{equation}

Using this disintegration, we write, 
\begin{equation}
\label{eq:decomp}
\begin{split}
|f(\psi)| & = \left| \int_{\xi \in \Xi} \int_{W_\xi} f \, \psi \, g_\xi \, |W_\xi|^{-1} dm_{W_\xi} d\hatmusrb(\xi) \right| \\
& \le \int_{\xi \in \Xi} |f|_w |\psi|_{\cC^1(W_\xi)} |g_\xi|_{\cC^1(W_\xi)} |W_\xi|^{-1} d\hatmusrb(\xi) \\
& \le C_g |f|_w \big(|\psi|_\infty + H^1_{\widetilde \cW^s}(\psi) \big) \int_{\xi \in \Xi} |W_\xi|^{-1} d\hatmusrb(\xi) \, .
\end{split}
\end{equation}
To bound this last integral, we will apply some results of \cite{chernov pw}, which studies
hyperbolic maps with singularities in an axiomatic context (Assumptions (H.1)-(H.5) in that paper),
which include the class of maps in the present paper, in addition to many dispersing and
semi-dispersing billiards.  Indeed,
the final integral in \eqref{eq:decomp} is precisely the $\cZ$-function, $\cZ_1(\cF)$, defined in \cite[eq.~(4.7)]{chernov pw} 
which governs the average length of stable manifolds in the family $\cF$.  (See also
\cite[Exercise~7.15 and Proposition~7.17]{chernov book} for a similar application of these
ideas.)  The parameters $p$ and $q$ in \cite{chernov pw} are both equal to 1 in our context,
due to our property (P1) and Convention~\ref{convention: n_0=1}, 
which imply that  $T$ satisfies the one-step expansion condition, \cite[Condition~(H.5)]{chernov pw} with parameter $q=1$,
\begin{equation}
\label{eq:one step}
\sup_{W \in \cW^s} \sum_{V_i \subset T^{-1}W} \left( \frac{|W|}{|V_i|} \right)^q \, \frac{|TV_i|}{|W|} 
 \le K_1 \Lambda^{-1} \le \rho <1\, ,
\end{equation}
where $V_i$ are the maximal, connected components
of $T^{-1}W$.  The required bound on $\cZ_1(\cF)$ follows from \cite[Lemma~4]{chernov pw}
(again with $q=1$) since $\musrb$ is obtained as the limit of standard pairs with finite
valued $\cZ$-function.
\end{proof}


\begin{lemma}
\label{lem:include}
There is a sequence of continuous inclusions,
\[
\cC^1(M) \hookrightarrow \cB \hookrightarrow \cB_w \hookrightarrow (\cC^\alpha(\cW^s))^* \, .
\]
The first two inclusions are injective.
\end{lemma}

\begin{proof}
The continuity of the first inclusion follows from Lemma~\ref{lem:piece}
and its injectivity is obvious.  The continuity of the second inclusion follows from
$| \cdot |_w \le \| \cdot \|_s$.  Its injectivity is a result of the fact that we have
defined $\| \cdot \|_s$ with respect to $\cC^\alpha(W)$ rather than $\tilde \cC^\alpha(W)$,
and $\cC^1(W)$ is dense in $\cC^\alpha(W)$.
Finally, the continuity of the third inclusion follows from Lemma~\ref{lem:embed}.
\end{proof}

By adding an additional weight to the weak norm, one can make the third inclusion
in Lemma~\ref{lem:include} injective as well (see for example \cite[Lemma~3.8]{demzhang14}), but
we will not need this property here.  Our final lemma in this section is essential for proving
the quasi-compactness of $\cL$ on $\cB$.

\begin{lemma}
\label{lem:compact}
The unit ball of $\cB$ is compactly embedded in $\cB_w$.
\end{lemma}

\begin{proof}
The lemma follows from \cite[Lemma~3.5]{demers liv}.  The fact that \cite[Lemma~3.5]{demers liv}
uses the family of admissible curves $\hW^s$ while we use the smaller
set $\cW^s \subset \hW^s$ does not affect the argument since the family of functions defining
 $\cW^s$ in each chart is still
compact in the $\cC^1$-metric.
\end{proof}


\subsection{Growth Lemmas}
\label{sec:growth}

In this section, we prove several growth lemmas which will be instrumental in 
establishing precise upper and lower bounds on the spectral radius of our transfer operator.
Many of the results in this subsection and the next parallel those of \cite[Section~5]{max}.

Given a curve $W \in \hW^s$, let $\cG_1(W)$ denote the maximal connected components
of $T^{-1}W$ on which $T$ is smooth, with long pieces subdivided so that they have length between $\delta_0/2$ and 
$\delta_0$.  In particular, elements of $\cG_1(W)$ must belong to a single element of $\cM_0^1$, i.e. to a single component $M_i^+$ of $M$. 
Inductively, define $\cG_n(W)$ to denote the collection of maximal connected
components of $T^{-1}V$, where $V \in \cG_{n-1}(W)$, again subdividing long pieces into
curves of length between $\delta_0/2$ and $\delta_0$.  We call $\cG_n(W)$, the $n$th generation
of $W$.

For each $n$, let $L_n(W)$ denote those elements of $\cG_n(W)$ having length at least
$\delta_0/3$.  Let $\cI_n(W)$ denote those elements $W_i \in \cG_n(W)$ such that for each
$0 \le k \le n-1$, 
$T^kW_i \subset V \in \cG_{n-k}(W)$ and $|V| < \delta_0/3$, i.e. $\cI_n(W)$ represents those
elements in $\cG_n(W)$ that have always been contained in a short element
of $\cG_{n-k}(W)$ from time $1$ to time $n$.

\begin{lemma}
\label{lem:growth}
There exists $C>0$ such that for all $W \in \hW^s$, and all $n \ge 0$,
\begin{itemize}
  \item[a)]  $\ds \# \cI_n(W) \le  K_1^n \le \rho^n \kappa^{\alpha_0 n} \Lambda^n \;$ ; \vspace{3 pt}
  \item[b)]  $\ds \# \cG_n(W) \le C \delta_0^{-1} \# \cM_0^n \;$ ; \vspace{3 pt}
  \item[c)]  $\ds \sum_{W_i \in \cG_n(W)} \frac{|W_i|^{1/p}}{|W|^{1/p}} \le C \delta_0^{-1+1/p} \kappa^{-n/p} (\# \cM_0^n)^{1-1/p} \;$ ; \vspace{3 pt}
  \item[d)] $\ds \# \cM_0^n \ge C \delta_0 \Lambda^n \,$ .
\end{itemize}
\end{lemma}

\begin{proof}
(a) This estimate follows from the fact that curves $W_i \in \cI_n(W)$
have always been contained in a short element of $\cG_{n-k}(W)$ for each $k$ between
0 and $n-1$.  Thus property (P1) (recalling also Convention~\ref{convention: n_0=1}) can be 
applied inductively in $k$ to each element of $\cI_{n-k}(W)$, yielding the
claimed bound on the cardinality of these elements.

\medskip
\noindent
(b) The bound is trivial since each element of $\cG_n(W)$ belongs by definition to one
element of $\cM_0^n$.  Since the stable diameter of each component of $\cM_0^n$ is
uniformly bounded in $n$, the connected components of $T^{-n}W$ are subdivided into
at most $C \delta_0^{-1}$ curves to form the elements of $\cG_n(W)$, for some uniform $C>0$.

\medskip
\noindent
(c) Note that for $W_i \in \cG_n(W)$, using \eqref{eq:exp def},
\[
|T^nW_i| = \int_{W_i} J_{W_i}T^n \, dm_{W_i} \ge |W_i| \kappa^n \, .
\]
Thus,
\[
\begin{split}
\sum_{W_i \in \cG_n(W)} \frac{|W_i|^{1/p}}{|W|^{1/p}}
& \le \kappa^{-n/p} \sum_{W_i \in \cG_n(W)} \frac{|T^nW_i|^{1/p}}{|W|^{1/p}}
\le \kappa^{-n/p} \left( \sum_{W_i \in \cG_n(W)} 1 \right)^{1-1/p} \\
& \le C \kappa^{-n/p} \delta_0^{-1+1/p} (\# \cM_0^n )^{1-1/p} \, ,
\end{split}
\]
where we have used the H\"older inequality and part (b) of the lemma.

\medskip
\noindent
(d)  Applying  part (b) of the lemma, we have
\[
|T^{-n}W| = \sum_{W_i \in \cG_n(W)} |W_i| \le \delta_0 \# \cG_n(W) \le C \# \cM_0^n \, .
\]
Then recalling \eqref{eq:one grow} and applying this to $W \in \cW^s$ with $|W| = \delta_0$ completes the proof of the lemma.
\end{proof}

Next we proceed to show that most elements of $\cG_n(W)$ are long, if the length scale
is chosen appropriately.  For $\delta \in (0, \delta_0)$ and $W \in \hW^s$, define
$\cG_n^\delta(W)$ to be the smooth components of $T^{-n}W$, with pieces longer than
$\delta$ subdivided
to have length between $\delta/2$ and $\delta$, i.e. $\cG_n^\delta(W)$ is defined precisely
like $\cG_n(W)$, but with $\delta_0$ replaced by $\delta$.  Define $L_n^\delta(W)$ to
be the set of curves in $\cG_n^\delta(W)$ having length at least $\delta/3$, and let
$S^\delta_n(W) = \cG_n^\delta(W) \setminus L_n^\delta(W)$.  Similarly, let 
$\cI_n^\delta(W)$ denote those elements of $S^\delta_n(W)$ that have no ancestors
of length at least $\delta/3$.

\begin{lemma}
\label{lem:most grow}
For all $\ve >0$, there exist $\delta \in (0, \delta_0)$ and $n_1 \in \mathbb{N}$ such that
for all $n \ge n_1$,
\[
\# L_n^\delta(W) \ge (1-\ve) \# G_n^\delta(W) \, , \quad \mbox{ for all $W \in \hW^s$ with
$|W| \ge \delta/3$.}
\]
\end{lemma}

\begin{proof}
Fix $\ve \in (0,1)$ and by Property (P1) and Remark~\ref{rem:iterate p1}, 
choose $n_1$ sufficiently large that $3 C_e^{-1} (K(n_1+\ell) + 1)\Lambda^{-n_1-\ell} < \ve/2$
for all $0 \le \ell \le n_1-1$, 
where $C_e \le 1$ is from \eqref{eq:one grow}.  Next, choose
$\delta>0$ sufficiently small that if $W \in \hW^s$ with $|W| \le \delta$, then $T^{-n}W$
comprises at most $K(n)+1$ smooth components of length at most $\delta_0$ for all
$n \le 2n_1$.

Now let $W \in \hW^s$ with $|W| \ge \delta/3$.  We shall prove that 
for $n \ge n_1$,
\[
\# S_n^\delta(W) \le \ve \# \cG_n^\delta(W) \, .
\]

For $n \ge n_1$, write $n = kn_1 + \ell$ for some $0 \le \ell < n_1$.  If $k=1$, the above
inequality follows immediately since there are at most $K(n_1+\ell)+1$
elements of $S^\delta_{n_1+\ell}(W)$ by choice of $\delta$, while by \eqref{eq:one grow},
$|T^{-n_1-\ell}W| \ge C_e \Lambda^{n_1 +\ell} |W| \ge C_e \Lambda^{n_1+\ell} \delta/3$.  Thus
$\cG_n^\delta(W)$ must contain at least $C_e \Lambda^{n_1+\ell}/3$ curves since
each has length at most $\delta$.  Thus,
\[
\frac{\# S_{n_1+\ell}^\delta(W)}{\# \cG_{n_1+\ell}^\delta(W)} \le 3 C_e^{-1} \frac{K(n_1+\ell)+1}{\Lambda^{n_1+\ell}}   < \frac{\ve}{2} \, ,
\]
 by assumption on $n_1$.

On the other hand, if $k >1$ then we split $n$ into $k-1$ blocks of length $n_1$ and one block 
of length $n_1+\ell$.  We group elements $W_i \in S^\delta_{kn_1+\ell}(W)$ by most
recent long ancestor $V_j \in L^\delta_{t n_1}(W)$:  $t$ is the greatest index $\le k-1$ such that
$T^{(k-t)n_1+\ell}W_i \in V_j$ and $V_j \in L_{t n_1}^\delta(W)$.  Note that we only
consider ancestors occurring in blocks of length $n_1$.  It is irrelevant for our estimate whether
$W_i$ has a long ancestor at an intermediate time.

Since each $|V_j| \ge \delta/3$, it follows that $\cG_{(k-t)n_1+ \ell}^\delta(V_j)$ must contain
at least $C_e \Lambda^{(k-t)n_1}/3$ curves of length at most $\delta$.  Thus using
Lemma~\ref{lem:growth}(a), we have
\begin{equation}
\label{eq:est short}
\begin{split}
\frac{\# S_{kn_1+\ell}^\delta(W)}{\# \cG_{kn_1+\ell}^\delta(W)}
& = \frac{\# \cI_{kn_1+\ell}^\delta(W)}{\# \cG_{kn_1+\ell}^\delta(W)}
+ \frac{\sum_{t=1}^{k-1} \sum_{V_j \in L_{tn_1}^\delta(W)} \# \cI_{(k-t)n_1+\ell}^\delta(V_j) }
{\# \cG_{kn_1+\ell}^\delta(W)} \\
& \le \frac{(K(n_1)+1)^k}{C_e \Lambda^{kn_1}/3} + \sum_{t=1}^{k-1} \frac{ \sum_{V_j \in L_{tn_1}^\delta(W)} (K(n_1)+1)^{k-t}}{\sum_{V_j \in L_{tn_1}^\delta(W)} C_e \Lambda^{(k-t)n_1}/3} \\
& \le 3 C_e^{-1} \sum_{t=1}^k (K(n_1)+1)^t \Lambda^{-tn_1} \le \sum_{t=1}^k \left( \frac{\ve}{2} \right)^t
< \ve \, .
\end{split}
\end{equation}
\end{proof}

The following corollary extends Lemma~\ref{lem:most grow} to arbitrarily short curves, and
is used in Lemma~\ref{lem:leafwise} to prove the positivity of our maximal eigenvector on all elements of $\cW^s$.

\begin{cor}
\label{cor:most grow}
There exists $C_2 > 0$ such that for any $\ve, \delta$ and $n_1$ as in Lemma~\ref{lem:most grow},
\[
\# L_n^\delta(W) \ge (1-2\ve) \# \cG_n^\delta(W) \, , \quad \forall W \in \hW^s, \; \forall n \ge C_2  n_1 \frac{|\log(|W|/\delta)|}{|\log \ve|} \, .
\]
\end{cor}

\begin{proof}
Fix $\ve, \delta$ and $n_1$ from Lemma~\ref{lem:most grow}.  Suppose 
$W \in \hW^s$ has $|W| < \delta/3$, and let $n > n_1$.  We decompose $\cG_n^\delta(W)$
as in Lemma~\ref{lem:most grow}, and estimate the second sum in \eqref{eq:est short}
precisely as before.  

The first term on the right hand side of \eqref{eq:est short}, $\# \cI_n^\delta(W)/ \# \cG_n^\delta(W)$,
is handled differently.  Let $n_2$ denote the least integer $\ell$ such that $\cG_\ell^\delta(W)$
contains at least one element of length $\delta/3$.  Since $|T^{-\ell}W| \ge C_e \Lambda^\ell |W|$
by \eqref{eq:one grow}, and $\cG_\ell^\delta(W) \le K_1^\ell$ by (P1) and 
Convention~\ref{convention: n_0=1}, as long as $|T^{-\ell}W| \le \delta_0$, at least one 
element of $\cG_\ell^\delta(W)$ must have
length at least $\frac{C_e \Lambda^\ell |W|}{K_1^\ell} \ge C_e \rho^{-\ell} |W|$.  Thus
\[
n_2 \le \frac{|\log (3C_e |W| \delta^{-1})|}{|\log \rho|} \, .
\]
Then calling $V$ the element of $\cG_{n_2}^\delta(W)$ having length at least $\delta/3$, we
have 
\[
\# \cG_n^\delta(W) \ge \# \cG_{n-n_2}^\delta(V) \ge C_e \Lambda^{n-n_2}/3 \, .
\]
Thus
\[
\frac{\# \cI_n^\delta(W)}{\# \cG_n^\delta(W)}
\le \frac{3 (K(n_1)+1)^{\lfloor n/n_1 \rfloor}}{C_e \Lambda^n} \Lambda^{n_2} \le \left( \frac{\ve}{2} \right)^{\lfloor n/n_1 \rfloor} \Lambda^{n_2} \, .
\]
Finally, since $n_2 = \mathcal{O}(|\log (|W|/\delta)|)$, we may choose $C_2$ sufficiently large, that 
if $n \ge C_2 n_1 \frac{|\log (|W|/\delta)|}{|\log \ve|}$, then the quantity on the right is at most $\ve$,
completing the proof of the corollary.
\end{proof}

Choosing $\ve = 1/3$, we let $\delta_1>0$ and $n_1$ be the corresponding quantities
from Lemma~\ref{lem:most grow}.  Fixing this choice of $\delta_1$ and $n_1$, we have
\begin{equation}
\label{eq:delta_1}
\# L_n^{\delta_1}(W) \ge \tfrac 23 \# \cG_n^{\delta_1}(W), \quad
\mbox{ for all $W \in \hW^s$ with $|W| \ge \delta_1/3$ and all $n \ge n_1$.}
\end{equation}

Our next lemma shows that a positive fraction of elements of $\cM_0^n$ and $\cM_{-n}^0$ have length
at least $\delta_1$ in some direction.  This will be essential to establishing the lower bounds
of Section~\ref{sec:lower}.  For $A \subset M$, let $\diam^s(A)$ denote
the stable diameter of $A$, i.e.
the length of the longest stable curve in $A$.  Similarly, define the unstable diameter
$\diam^u(A)$ to be the length of the longest unstable curve in $A$.

The boundary of the partition defined by $\cM_{-n}^0$ is comprised of unstable
curves belonging to $\cS^-_n = \cup_{i=0}^{n-1} T^i(\cS^-)$.  Similarly, $\partial \cM_0^n$
is comprised of the stable curves, $\cS^+_n = \cup_{i=0}^{n-1} T^{-i}(\cS^+)$.  
In what follows, we will find it convenient to invoke Convention~\ref{con:pointwise} regarding
the definition of $T^{\pm 1}$ on each smooth component of $\cS^{\pm}$.
Let $L_u(\cM_{-n}^0)$
denote those elements of $\cM_{-n}^0$ whose unstable diameter is at least
$\delta_1/3$, and let $L_s(\cM_0^n)$ denote those elements of $\cM_0^n$ whose stable diameter
is at least $\delta_1/3$.  The following lemma is the analogue of Lemma~\ref{lem:most grow}
for these dynamically defined partitions.

\begin{lemma}
\label{lem:long elements}
There exist $C_{n_1}>0$ and $n_3 \ge n_1$ such that for all $n \ge n_3$,
\[
\# L_s(\cM_0^n) \ge C_{n_1} \delta_1 \# \cM_0^n
\quad \mbox{and} \quad
\# L_u(\cM_{-n}^0) \ge C_{n_1} \delta_1 \# \cM_{-n}^0 \, .
\]
\end{lemma}

\begin{proof}
We prove the bound for $L_s(\cM_0^n)$.   In order to prove the lemma, we will use the fact that  
the boundary of $\cM_0^n$ is  the set $\cup_{j=0}^{n-1} T^{-j}\cS^+$.

Let $S_s(\cM_0^n)$ denote the elements of $\cM_0^n$ whose stable diameter is less than
$\delta_1/3$.  We have $\cM_0^n = L_s(\cM_0^n) \cup S_s(\cM_0^n)$.  Similarly, let
$S_s(T^{-j} \cS^+)$ denote the set of stable curves in $T^{-j} \cS^+$ whose length is less than
$\delta_1/3$.  

The following sublemma will prove useful for establishing key claim in the proof.

\begin{sublem}
\label{sub:cross}
If a smooth stable curve $V_i \in T^{-i}\cS^+$ intersects a smooth curve $V_j \subset T^{-j}\cS^+$
for $i<j$, then $V_j$ must terminate on $V_i$.
\end{sublem}

\begin{proof}[Proof of Sublemma~\ref{sub:cross}]
Suppose such an intersection occurs for $j >i$.  
Then $T^{i+1}(V_i) \subset \cS^-$ is an unstable curve,
while $T^{i+1}(V_j) \subset \cS^+_{j-i-1}$ is a stable curve. Thus
$T^{i+1}(V_j)$ must cross $\cS^-$ transversally, and so $T^{i}(V_j)$ will be split into at 
least two smooth components since $\cS^-$ is the singularity set for $T^{-1}$.  This implies
that $V_j$ cannot be a single smooth curve.
\end{proof}

Using the sublemma, we establish the following claim:  
\begin{equation}
\label{eq:claim}
\# S_s(\cM_0^n) \le 2 \sum_{j=0}^{n-1} \# S_s(T^{-j}\cS^+) + B_1 n \, ,
\end{equation}
for some $B_1 >0$.
According to the sublemma, if $A \in S_s(\cM_0^n)$, then either
$\partial A$ contains a short curve in $T^{-j}\cS^+$ or
$\partial A$ contains an intersection point of two curves in $T^{-j}\cS^+$,  for some $0 \le j \le n-1$.
But intersections of curves within $T^{-j}\cS^+$ are images of intersections of curves
within $\cS^+$, and the cardinality of cells created by such intersections is bounded by some
uniform constant $B_1>0$ depending only on $\cS^+$.  Since each short curve in
$T^{-j}\cS^+$ belongs to the boundary of at most two elements of $S_s(\cM_0^n)$, the claim
follows.

Now, we subdivide $\cS^+$ into $\ell_0$ smooth curves $V_i$ of length between $\delta_1/3$
and $\delta_1$.  For $j \ge n_1$, recalling the notation $S_j^{\delta_1}(V_i)$
for the short elements of the $j$th generation $\cG_j^{\delta_1}(V_i)$ 
of subcurves in $T^{-j}V_i$, we have by \eqref{eq:delta_1},
\begin{equation}
\label{eq:short bound}
\# S_s(T^{-j}\cS^+) = \sum_{i=1}^{\ell_0} \# S_j^{\delta_1}(V_i) 
\le \tfrac 13 \sum_{i=1}^{\ell_0} \# L_j^{\delta_1}(V_i) \, .
\end{equation}
Next, using \eqref{eq:short bound}, we estimate the sum over $j$ in \eqref{eq:claim} by splitting it over
two parts,
\begin{equation}
\label{eq:j split}
\# S_s(\cM_0^n) \le B_1n + 2 \sum_{j=0}^{n_1-1} \# S_s(T^{-j}\cS^+)
+ \tfrac 23 \sum_{j=n_1}^{n-1} \sum_{i=0}^{\ell_0} \# L_j^{\delta_1}(V_i) \, .
\end{equation}
The cardinality of the first sum up to $n_1-1$ is bounded by some constant $\bar C_{n_1}$
depending only on the map $T$ and $n_1$, but independent of $n$.

Next, we wish to relate $\# L_j^{\delta_1}(V_i)$ to $\# L_s(\cM_0^n)$ for $j \ge n_1$.
Note that if $V' \in L_j^{\delta_1}(V_i)$, then
$|T^{n-j}V'| \ge C \Lambda^{n-j} \delta_1/3$, so that $\# \cG_{n-j}^{\delta_1}(V') \ge C \Lambda^{n-j}/3$.

Now for each $j$ such that $n_1 \le j \le n-1-n_1$, and $V' \in L_j^{\delta_1}(V_i)$, we may apply
\eqref{eq:delta_1}, so that
\begin{equation}
\label{eq:relate j}
\# L_{n-1}^{\delta_1}(V_i) \ge \sum_{V' \in L_j^{\delta_1}(V_i)} \# L_{n-1-j}^{\delta_1}(V') \ge C' \Lambda^{n-1-j} \# L_j^{\delta_1}(V_i) \, .
\end{equation}
For $j > n-n_1$, we compare $L_j^{\delta_1}(V_i)$ with $L_{n-1}^{\delta_1}(V_i)$.  
Since $K_1 < \Lambda$, there is at least one element of $L_{j+1}^{\delta_1}(V_i)$
for each element of $L_j^{\delta_1}(V_i)$.  Applying this inductively to $j$, we conclude,
\[
\# L_{n-1}^{\delta_1}(V_i) \ge  \# L_j^{\delta_1}(V_i) \, .
\]
Putting together this estimate with \eqref{eq:relate j} in 
\eqref{eq:j split}, we estimate,
\begin{equation}
\label{eq:almost}
\begin{split}
\# S_s(\cM_0^n) & \le B_1 n + \bar C_{n_1} + \sum_{j=n_1}^{n-1-n_1} C \Lambda^{j+1-n} \# L_s(T^{-n+1} \cS^+) + \sum_{j=n-n_1}^{n-1}   \# L_s(T^{-n+1}\cS^+) \\
& \le B_1 n + \bar C_{n_1} + C \delta_1^{-1} \# L_s(\cM_0^n)
+ n_1 C \delta_1^{-1} \# L_s(\cM_0^n) \, ,
\end{split}
\end{equation}
where in the second line we have used the fact that $\# L_s(T^{-n+1} \cS^+) \le C \delta_1^{-1} \# L_s(\cM_0^n)$, which follows from Sublemma~\ref{sub:cross}.

Finally, since $\# \cM_0^n = \# L_s(\cM_0^n) + \# S_s(\cM_0^n)$, we estimate,
\[
\# L_s(\cM_0^n) \ge \frac{ \# \cM_0^n - \bar C_{n_1} - B_1 n}{1 + C \delta_1^{-1}(1+n_1)} \, .
\]
Since $\# \cM_0^n \ge C\delta_0 \Lambda^n$ by Lemma~\ref{lem:growth}(d) and $n_1$ is fixed,
we may choose $n_2 \in \mathbb{N}$ such that $\# \cM_0^n - \bar C_{n_1} - B_1 n \ge \frac 12 \# \cM_0^n$,
for all $n \ge n_2$.  We conclude that there exists $C_{n_1} > 0$ such that for
$n \ge n_2$, $\# L_s(\cM_0^n) \ge C_{n_1} \delta_1 \# \cM_0^n$, completing the proof of the lemma
for $L_s(\cM_0^n)$.

The lower bound for $\# L_u(\cM_{-n}^0)$ follows
similarly, using the fact that (P1) also allows us to control the evolution of unstable 
curves under $T^n$ by controlling the complexity of $\cS_n^+$.
Note that the analogue of Lemma~\ref{lem:most grow} holds for forward iterates of unstable
curves using precisely the same proof.  The constant $\kappa$ does not appear in this argument,
i.e. the fact that the rate of expansion has a maximum is not needed for the proof.
\end{proof}


\subsection{Lower bounds on growth}
\label{sec:lower}

The prevalence of long pieces established in Lemmas~\ref{lem:most grow} and 
\ref{lem:long elements} have the following important consequences.

\begin{lemma}
\label{lem:lower}
Let $\delta_1$ be the length scale from \eqref{eq:delta_1}.  There exists $c_0>0$, depending on
$\delta_1$, such that
for all $W \in \hW^s$ with $|W| \ge \delta_1/3$ and $n\ge 1$, we have
$\# \cG_n(W) \ge c_0 \# \cM_0^n$.
\end{lemma}

This lemma, in turn, implies the supermultiplicativity property for $\# \cM_0^n$.

\begin{proposition}
\label{prop:super}
There exists $c_1>0$ such that for all $j, n \in \mathbb{N}$ with $j \le n$, it holds,
\[
\# \cM_0^n \ge c_1 \# \cM_0^{n-j} \# \cM_0^j \, .
\]
\end{proposition}

In order to establish Lemma~\ref{lem:lower}, we recall the construction of Cantor rectangles.
For $x \in M$, let $W^s(x)$ and $W^u(x)$ denote the maximal smooth components of the local stable and
unstable manifolds of $x$ (which, by definition, belong to a single domain $M_i^+$).

We begin by defining a solid rectangle $D \subset M$ to be a closed region whose boundary
comprises exactly two stable manifolds and two unstable manifolds of positive length.
Given such a region $D$, define the {\em locally maximal Cantor rectangle $R$ in $D$}
to be the union of all points in $D$ whose local stable and unstable manifolds completely
cross $D$.  Locally maximal Cantor rectangles are endowed with a natural product structure:
for any $x,y \in R$, $W^u(x) \cap W^s(y)$ belongs to $R$.  Such rectangles are closed,
so their boundary coincides with the boundary of $D$.  In this case, we write
$D = D(R)$ to denote the fact that $D$ is the smallest solid rectangle containing $R$.

Following \cite{Li1}, for a Cantor rectangle $R$, we call the {\em core} of $R$ to be 
$R \cap D_{1/4}$, where $D_{1/4}$ is an approximately concentric rectangle in 
$D(R)$ with side lengths
$1/4$ the side lengths of $D$. 

For a locally maximal Cantor rectangle $R$, we say that a stable (respectively unstable) curve $W$
{\em properly crosses} $R$ if $W$ intersects the rectangle $D_{1/4}(R)$, 
but does not terminate in $D(R)$, and 
$W$ does not cross either of the stable (resp. unstable) boundaries of both $D(R)$ and
$D_{1/4}(R)$. 

\begin{proof}[Proof of Lemma~\ref{lem:lower}]
Applying \cite[Theorem~4.10]{Li1}, we may choose locally maximal Cantor 
rectangles $\cR_{\delta_1} = \{ R_1, \cdots, R_k \}$,
with $\musrb(R_i)>0$, whose stable and unstable boundaries have length at most 
$\frac{1}{10} \delta_1$ such that
any stable or unstable curve of length at least $\delta_1/3$ properly crosses at least one of 
them.\footnote{Once a Cantor rectangle of some size is constructed around $\musrb$-almost-every 
$x \in M$, the existence of such a finite family for any fixed length scale $\delta_1$
follows from the compactness of the set of stable (and also unstable) curves of length $\ge \delta_1/3$
in the Hausdorff metric, as in \cite[Lemma~7.87]{chernov book}.}
Furthermore, we may choose the rectangles sufficiently small that both $R_i$ and 
$R_i \cap D_{1/4}(R_i)$ have positive $\musrb$-measure for each $i$.
The number of rectangles $k$ depends on $\delta_1$. 

For brevity, denote by $R_i^* = R_i \cap D_{1/4}(R_i)$, the core of $R_i$.  
Due to the mixing property of $(T, \musrb)$, 
there exist $\ve>0$ and $n_4 \in \mathbb{N}$
such that for all $n \ge n_4$, and all $1 \le i,j, \le k$, $\musrb(R_i^* \cap T^{-n}R_j) \ge \ve$.

We claim that for each $n$, at least one Cantor rectangle $R_i \in \cR_{\delta_1}$ is fully
crossed in the unstable direction by at least $\frac 1k \#L_u(\cM_{-n}^0)$ elements of
of $\cM_{-n}^0$.  This is because if $A \in \cM_{-n}^0$, then $\partial A$ is comprised of
unstable curves belonging to $\cS_n^-$.  Since unstable manifolds cannot be cut under
iteration by $T^{-n}$, $\cS_n^-$ cannot intersect the unstable boundaries of $R_i$.
Thus if $A \cap R_i \neq \emptyset$, then either $\partial A$ terminates inside $R_i$ or
$A$ fully crosses $R_i$.  This implies that elements of $L_u(\cM_{-n}^0)$ fully cross at least
one $R_i$, and so at least one $R_i$ must be fully crossed by at least $\frac 1k$ such
elements.

With the claim established, for each $n$, let $R_{i_n}$ denote a Cantor rectangle that
is fully crossed by at least $\frac 1k \# L_u(\cM_{-n}^0)$ elements of $\cM_{-n}^0$.

Now take $W \in \hW^s$ with $|W| \ge \delta_1/3$.  By construction, there exists
$R_j \in \cR_{\delta_1}$ such that $W$ properly crosses $R_j$ in the stable direction.  For
each $n \in \mathbb{N}$, using mixing, we have $\musrb(R_{i_n}^* \cap T^{-n_4}R_j) \ge \ve$.
By \cite[Lemma~4.13]{Li1}, there is a curve $V \in \cG_{n_4}^{\delta_1}(W)$ that 
properly crosses $R_{i_n}$
in the stable direction.  By choice of $R_{i_n}$, this implies that 
$\# \cG_n(V) \ge \frac 1k \# L_u(\cM_{-n}^0)$.   
Thus,
\[
\# \cG_{n+n_4}(W) \ge \tfrac 1k \# L_u(\cM_{-n}^0) \implies 
\# \cG_n(W) \ge \tfrac{C'}{k} \# L_u(\cM_{-n}^0 ) \, ,
\]
where $(C')^{-1} = C \delta_0^{-1} \# \cM_0^{n_4}$ since 
$\# \cG_{n + n_4}(W) \le C \delta_0^{-1} \# \cM_0^{n_4} \# \cG_n(W)$ by Lemma~\ref{lem:growth}(b).

Finally, by Lemma~\ref{lem:long elements},
$\# L_u(\cM_{-n}^0) \ge C_{n_1} \delta_1 \# \cM_{-n}^0$, which proves the
lemma for $n \ge \max \{ n_3, n_4 \}$ since $\# \cM_{-n}^0 = \# \cM_0^n$.  The lemma
extends to all $n \in \mathbb{N}$ by possibly reducing the constant $c_0$ since there
are only finitely many values to correct for.
\end{proof}

\begin{proof}[Proof of Proposition~\ref{prop:super}]
Recall that since $T^{-j}(\cS^-_j \cup \cS^+_{n-j}) = \cS^+_n$, there is a one-to-one correspondence
between elements of $\cM_{-j}^{n-j}$ and $\cM_0^n$ for each $j < n$.  Thus
$\# \cM_0^n = \# \cM_{-j}^{n-j}$, and this latter partition is obtained by taking the maximal 
connected components of $\cM_{-j}^0 \bigvee \cM_0^{n-j}$.

To prove the lemma, we will show that a positive fraction, independent of $j$ and $n$,
of elements of $\cM_0^{n-j}$ intersect a positive fraction of elements of $\cM_{-j}^0$.
Recall that $L_u(\cM_{-j}^0)$ denotes those elements of $\cM_{-j}^0$ with unstable
diameter of length at least $\delta_1/3$ while $L_s(\cM_0^{n-j})$ denotes those elements
of $\cM_0^{n-j}$ with stable diameter of length at least $\delta_1/3$.

If $A \in L_s(\cM_0^{n-j})$ and $V \subset A$ is a stable curve with $|V| \ge \delta_1/3$, then
$\# \cG_j(V) \ge c_0 \# \cM_0^j$ by Lemma~\ref{lem:lower}.  Remark that up to subdivision
of long pieces, each component of $\cG_j(V)$ corresponds to one component of 
$V \setminus \cS^-_j$.  Thus $V$ intersects at least $c_0 \# \cM_0^j = c_0 \# \cM_{-j}^0$ elements
of $\cM_{-j}^0$.  Applying this estimate to each $A \in L_s(\cM_0^{n-j})$, we obtain
\[
\# \cM_0^n \ge \# L_s(\cM_0^{n-j}) \cdot c_0 \# \cM_0^j \ge C_{n_1}\delta_1 c_0 \# \cM_0^{n-j} \# \cM_0^j \, ,
\]
where we have applied Lemma~\ref{lem:long elements} in the second inequality.
This proves the lemma when $n-j \ge n_3$.  For $n-j \le n_3$, since $\# \cM_0^{n-j} \le \# \cM_0^{n_3}$,
we obtain the lemma by possibly decreasing the value of $c_1$ since there are only finitely many
values to correct for.
\end{proof}

\begin{cor}
\label{cor:upper M}
For all $n \in \mathbb{N}$, $\# \cM_0^n \le 2 c_1^{-1} e^{n h_*}$, where $c_1>0$ is from
Proposition~\ref{prop:super}.
\end{cor}

\begin{proof}
The proof follows using Proposition~\ref{prop:super}, precisely as in
\cite[Proposition~4.6]{max}.
\end{proof}


\section{Spectral Properties of $\cL$}
\label{sec:spec}

In this section, we prove the following theorem.

\begin{theorem}
\label{thm:spectral}
The operator $\cL$ acting on $\cB$ is quasi-compact, with spectral radius equal to
$e^{h_*}$ and essential spectral radius bounded by $\max \{ \Lambda^{-\beta}, \rho \} e^{h_*}$.

Since $T$ is topologically mixing, $\cL$ has a spectral gap:  $e^{h_*}$ is a simple eigenvalue (multiplicity 1 and no Jordan blocks)
and the rest of the spectrum of $\cL$ is contained in a disk of radius strictly smaller than $e^{h_*}$.

Let $\nu_0 \in \cB$ be an eigenfunction for eigenvalue $e^{h_*}$ defined by
\[
\nu_0  := \lim_{n \to \infty} \frac 1n \sum_{k=0}^{n-1} e^{-k  h_* } \cL^k 1 \, .
\]
Then $\nu_0 \neq 0$ is a non-negative Radon measure on $M$.
\end{theorem}

The quasi-compactness of $\cL$ is proved in Lemma~\ref{lem:radius}, following the Lasota-Yorke
inequalities of Proposition~\ref{prop:LY}.  The fact that $\cL$ has a spectral gap is proved
in Lemma~\ref{lem:gap}, while the characterization of $\nu_0$ is proved in 
Lemma~\ref{lem:peripheral}.


\subsection{Lasota-Yorke Inequalities}
\label{sec:LY}

The following proposition is the key component in establishing the quasi-compactness
of $\cL$.

\begin{proposition}
\label{prop:LY}
There exists $C>0$ such that for all $n \ge 0$ and $f \in \cB$,
\begin{eqnarray}
|\cL^n f |_w & \le & C \delta_0^{-1} (\# \cM_0^n) |f|_w \, , \label{eq:weak norm} \\
\| \cL^n f \|_s & \le & C \delta_0^{-2} (\# \cM_0^n) \big( (\Lambda^{-\alpha n} + \rho^n ) \| f \|_s +   \kappa^{-n/p} |f|_w \big) \, ,
\label{eq:stable norm} \\
\| \cL^n f \|_u & \le & C \delta_0^{-1} (\# \cM_0^n) (\Lambda^{-\beta n} \| f \|_u + \kappa^{-n/p}\| f \|_s )  \, .
\label{eq:unstable norm} 
\end{eqnarray}
\end{proposition}

By density, it suffices to prove the proposition for $f \in \cC^1(M)$.


\subsubsection{Weak norm bound}

Take $f \in \cC^1(M)$, $W \in \cW^s$ and $\psi \in \cC^1(W)$, $|\psi|_{\cC^1(W)} \le 1$.
Recalling that $\cG_n(W)$ denotes the decomposition of $T^{-n}W$ into elements of $\cW^s$, we
estimate for $n \ge 1$,
\[
\int_W \cL^n f \, \psi \, dm_W = \sum_{W_i \in \cG_n(W)}
\int_{W_i} f  \,  \psi \circ T^n \, dm_{W_i}
\le |f|_w \sum_{W_i \in \cG_n(W)}  |\psi \circ T^n|_{C^1(W_i)} \, ,
\]
where we have applied the weak norm of $f$ to the integral on each $W_i$.
Next, using the uniform contraction of $T$ along stable curves, we have
\begin{equation}
\label{eq:holder}
\frac{|\psi \circ T^n(x) - \psi \circ T^n(y)|}{d_{W_i}(x,y)} 
= \frac{|\psi \circ T^n(x) - \psi \circ T^n(y)|}{d_W(T^nx, T^ny)} \frac{d_W(T^nx, T^ny)}{d_{W_i}(x,y)}
\le C |J^sT^n|_{\cC^0(W_i)} H^1_W(\psi) \, ,
\end{equation}
for some uniform constant $C>0$, using \eqref{eq:exp def}.  
Then since $|\psi \circ T^n|_{\cC^0(W_i)} \le |\psi |_{\cC^0(W)}$, we have
$|\psi \circ T^n|_{\cC^1(W_i)} \le C |\psi|_{\cC^1(W)} \le C$.  Finally, applying
Lemma~\ref{lem:growth}(b) to the sum over $\cG_n(W)$ and taking the supremum over $\psi \in \cC^1(W)$ and 
$W \in \cW^s$ completes the proof of \eqref{eq:weak norm}.


\subsubsection{Strong stable norm bound}

Let $f \in \cC^1(M)$, $W \in \cW^s$ and $\psi \in \cC^\alpha(W)$ with 
$|\psi|_{\cC^\alpha(W)} \le |W|^{-1/p}$.  Let $n \ge 1$.  For each $W_i \in \cG_n(W)$,
define $\bpsi_i = |W_i|^{-1} \int_{W_i} \psi \circ T^n \, dm_{W_i}$.  Proceeding as before, we estimate
\begin{equation}
\label{eq:stable split}
\int_W \cL^n f \, \psi \, dm_W = \sum_{W_i \in \cG_n(W)} \int_{W_i} f \, (\psi \circ T^n - \bpsi_i) \, dm_{W_i} + \sum_{W_i \in \cG_n(W)} \bpsi_i \int_{W_i} f \, dm_{W_i} \, .
\end{equation}
To each term in the first sum on the right hand side, we apply the strong stable norm,
\[
\int_{W_i} f \, (\psi \circ T^n - \bpsi_i)
\le  \| f \|_s |W_i|^{1/p} |\psi \circ T^n - \bpsi_i|_{\cC^\alpha(W_i)}
\le C \| f \|_s  \frac{|W_i|^{1/p}}{|W|^{1/p}} |J^sT^n|_{\cC^0(W_i)}^\alpha \, ,
\]
where we have applied the analogous estimate to \eqref{eq:holder} to the difference
$\psi \circ T^n - \bpsi_i$ with the exponent $\alpha$. Since $\alpha > 1/p$, using
bounded distortion \eqref{eq:distortion}, we estimate
\[
|W_i|^{1/p} |J^sT^n|_{\cC^0(W_i)}^\alpha \le C |T^nW_i|^{1/p} \Lambda^{-n(\alpha - 1/p)} \, .
\]
Finally, summing over $W_i$, we obtain,
\begin{equation}
\label{eq:first stable}
\begin{split}
\sum_{W_i \in \cG_n(W)} & \int_{W_i} f \, (\psi \circ T^n - \bpsi_i)
\le C \| f\|_s \Lambda^{-n(\alpha - 1/p)} \sum_{W_i \in \cG_n(W)} \frac{|T^nW_i|^{1/p}}{|W|^{1/p}} \\
& \le C \| f\|_s \Lambda^{-n(\alpha - 1/p)} \left( \sum_i \frac{|T^nW_i|}{|W|} \right)^{1/p}
( \# \cG_n(W_i))^{1- 1/p} \\
& \le C \delta_0^{-1+1/p} \| f\|_s \Lambda^{-n(\alpha - 1/p)} (\# \cM_0^n)^{1-1/p} 
\le C \delta_0^{-1}  \Lambda^{- \alpha n } \| f \|_s \# \cM_0^n \, ,
\end{split}
\end{equation}
where in the second line we have used the H\"older inequality and in the third we have
used Lemma~\ref{lem:growth}(b) and (d).

Next, we estimate the second sum in \eqref{eq:stable split}.  For this estimate, we group
$W_i \in \cG_n(W)$ by most recent long ancestor as follows.  Recall that
$L_k(W)$ denotes those elements of $\cG_k(W)$ whose length is at least $\delta_0/3$.
If $V_j \in L_k(W)$ is such that $T^{n-k}(W_i) \subset V_j$ and $k \le n$ is the largest
such index with this property, then we say that $V_j$ is the most recent long ancestor of
$W_i$.  Let $\cI_{n-k}(V_j)$ denote those elements of $\cG_n(W)$ whose most recent long
ancestor is $V_j$.  If no such ancestor exists, then $W_i \in \cI_n(W)$.  Thus,
\[
 \sum_{W_i \in \cG_n(W)} \bpsi_i \int_{W_i} f \, dm_{W_i}
 = \sum_{k=1}^n \sum_{V_j \in L_k(W)} \sum_{W_i \in \cI_{n-k}(V_j)} \bpsi_i \int_{W_i} f \, dm_{W_i} 
 + \sum_{W_i \in \cI_n(W)} \bpsi_i \int_{W_i} f \, dm_{W_i}  \, .
 \]
We use the strong stable norm to estimate the terms in $\cI_n(W)$,
\begin{equation}
\label{eq:short stable}
\begin{split}
\sum_{W_i \in \cI_n(W)} & \bpsi_i \int_{W_i} f \, dm_{W_i}
\le \| f\|_s \sum_{W_i \in \cI_n(W)} \frac{|W_i|^{1/p}}{|W|^{1/p}}
\le \| f \|_s \kappa^{-n/p} \sum_{W_i \in \cI_n(W)} \frac{|T^nW_i|^{1/p}}{|W|^{1/p}} \\
& \le \|f \|_s \kappa^{-n/p} K_1^{n(1-1/p)}
\; \le \; \| f\|_s \kappa^{-n/p} \rho^n \kappa^{\alpha_0 n} \Lambda^n \;  \le \; \| f\|_s \rho^n C\delta_0^{-1} \# \cM_0^n \, ,
\end{split}
\end{equation}
where we have used \eqref{eq:exp def} for the second inequality, the H\"older inequality and
Lemma~\ref{lem:growth}(a) for the third and fourth inequalities, and the fact that 
$\alpha_0 \ge 1/p$ (from \eqref{eq:restrict}) and Lemma~\ref{lem:growth}(d) for the last inequality.

For the remainder of the terms, we use the weak norm of $f$, and sum using
Lemma~\ref{lem:growth}(a) from time $k$ to time $n$,
\[
\begin{split}
\sum_{k=1}^n & \sum_{V_j \in L_k(W)} \sum_{W_i \in \cI_{n-k}(V_j)} \bpsi_i \int_{W_i} f \, dm_{W_i}
\le \sum_{k=1}^n \sum_{V_j \in L_k(W)} \sum_{W_i \in \cI_{n-k}(V_j)} |W|^{-1/p} |f|_w \\
& \le \sum_{k=1}^n \sum_{V_j \in L_k(W)} 3 \delta_0^{-1/p} K_1^{n-k} \frac{|V_j|^{1/p}}{|W|^{1/p}} |f|_w
\le \sum_{k=1}^n C \delta_0^{-1} K_1^{n-k} \kappa^{-k/p} (\# \cM_0^k)^{1-1/p} |f|_w \\
& \le C \delta_0^{-1} |f|_w \kappa^{-n/p} \sum_{k=1}^n \rho^{n-k} \kappa^{\alpha_0(n-k)} \Lambda^{n-k} 
\# \cM_0^k \le C \delta_0^{-2} c_1^{-1} \kappa^{-n/p} \# \cM_0^n |f|_w \, ,
\end{split}
\]
where we have used Lemma~\ref{lem:growth}(c) to sum over $V_j \in L_k(W)$, as well as
the fact that 
\[
\Lambda^{n-k} \# \cM_0^k \le C \delta_0^{-1} \# \cM_0^{n-k} \# \cM_0^k \le C \delta_0^{-1} c_1^{-1} \# \cM_0^n \, , 
\]
by Proposition~\ref{prop:super}.
Putting this estimate together with \eqref{eq:short stable} and \eqref{eq:first stable} in
\eqref{eq:stable split} yields,
\[
\int_W \cL^n f \, \psi \, dm_W \le C \delta_0^{-2}  \big( (\Lambda^{-\alpha n} + \rho^n)  \|f \|_s + 
\kappa^{-n/p} |f|_w \big) \# \cM_0^n \, ,
\]
and taking the appropriate suprema over $W$ and $\psi$ completes the proof of \eqref{eq:stable norm}.


\subsubsection{Strong unstable norm bound}

Let $f \in \cC^1(M)$ and $\ve \in (0, \ve_0)$.  Take $W^1, W^2 \in \cW^s$ with $d_{\cW^s}(W^1, W^2) \le \ve$,
and $\psi_k \in \cC^1(W^k)$ such that $|\psi_k|_{\cC^1(W^k)} \le 1$ and $d_0(\psi_1, \psi_2) = 0$.  For $n \ge1$, we
subdivide $\cG_n(W^k)$ into matched and unmatched pieces as follows.

To each $W^1_i \in \cG_n(W^1)$, we associate a family of vertical (in the chart) segments
$\{ \gamma_x \}_{x \in W^1_i}$ of length at most $C \Lambda^{-n} \ve$ such that if 
$\gamma_x$ is not cut by an element of $\cS_n^+$, its image $T^n\gamma_x$ will have length
$C\ve$ and will intersect $W^2$.  Due to the uniform transversality of stable and unstable cones,
such a segment $T^i \gamma_x$ will belong to the unstable cone for each $i = 0, \ldots, n$, and so
undergo the uniform expansion due to \eqref{eq:exp def}.

In this way, we obtain a partition of $W^1$ into intervals for which $T^n\gamma_x$ is not cut
and intersects $W^2$ and subintervals for which this is not the case.  This defines an analogous
partition of $T^{-n}W^1$ and $T^{-n}W^2$.  We call two curves $U^1_j \subset T^{-n}W^1$ 
and $U^2_j \subset T^{-n}W^2$ {\em matched} if they are connected by the foliation
$\gamma_x$ and their images under $T^n$ are connected by $T^n\gamma_x$.  
We call the remaining components of $T^{-n}W^k$ {\em unmatched} and denote them
by $V^k_i$.  With this decomposition, there is at most one matched piece and two unmatched
pieces for each $W^k_i \in \cG_n(W^k)$, and we may write $T^{-n}W^k = (\cup_j U^k_j) \cup (\cup_i V^k_i)$.

We proceed to estimate,
\begin{equation}
\label{eq:unstable split}
\left| \int_{W^1} \cL^n f \, \psi_1 - \int_{W^2} \cL^n f \, \psi_2 \right|
\le \sum_j \left| \int_{U^1_j} f \, \psi_1 \circ T^n - \int_{U^2_j} f \, \psi_2 \circ T^n \right|
+ \sum_{k,i} \left| \int_{V^k_i} f \, \psi_k \circ T^n \right| \, .
\end{equation}

We begin by estimating the contribution from unmatched pieces.  We say a curve
$V^1_i$ is created at time $j$, $1 \le j \le n$, if $j$ is the first time that $T^{n-j}V^1_i$
is not part of a matched curve in $T^{-j}W^1$.  Define,
\[
\cV_{j,\ell} = \{ i : V^1_i \mbox{ is created at time $j$ and $T^{n-j}V^1_i \subset W^1_\ell \in \cG_j(W^1)$} \} \, .
\]
Note that $\cup_{i \in \cV_{j,\ell}} V^1_i = W^1_\ell$.
Due to the expansion of $T$ in the unstable cone and the uniform transversality of
$\cS_j^-$ with the stable cone, it follows that $|W^1_\ell| \le C \Lambda^{-j} \ve$.
Now applying the strong stable norm to each such curve at the time it is created,
\begin{equation}
\label{eq:unmatched}
\begin{split}
\sum_i  \int_{V^1_i} f \, \psi_1 \circ T^n 
& = \sum_{j=1}^n \sum_{W^1_\ell \in \cG_j(W)} 
 \int_{W^1_\ell} \cL^{n-j} f \, \psi_i \circ T^{n-j} \\
 & \le  \sum_{j=1}^n \sum_{W^1_\ell \in \cG_j(W)} 
|W^1_\ell|^{1/p} \| \cL^{n-j} f \|_s |\psi \circ T^{n-j}|_{\cC^\alpha(W^1_\ell)} \\
& \le \sum_{j=1}^n \sum_{W^1_\ell \in \cG_j(W)} C \Lambda^{-j/p} \ve^{1/p} 
\delta_0^{-1} \kappa^{-(n-j)/p} (\# \cM_0^{n-j}) \| f \|_s \\
& \le C \delta_0^{-1} \ve^{1/p} \| f \|_s \kappa^{-n/p} \sum_{j=1}^n \Lambda^{-j/p} \# \cM_0^j \# \cM_0^{n-j} 
\le C \delta_0^{-1} \ve^{1/p} \| f \|_s \kappa^{-n/p} \# \cM_0^n \, ,
\end{split}
\end{equation}
where we have applied \eqref{eq:stable norm} in the second inequality (actually, a simpler
version suffices with no need to subtract the average of the test function on each $W_i$), 
and Proposition~\ref{prop:super} in the fourth.  A similar estimate holds over the curves $V^2_i$.

Next, we estimate the matched pieces.  Recall that according to our notation in Section~\ref{sec:admissible}
the curve $U^1_j$ is associated with the quadruple $(i_j, x_j, r_j, F^1_j)$ so that $F^1_j$
is defined in the chart $\chi_{i_j}$ and
$U^1_j = G(x_j, r_j, F^1_j)(I_{r_j})$.  By definition of our matching
process, it follows that $U^2_j = G(x_j, r_j, F^2_j)(I_{r_j})$ for some function
$F^2_j$ defined in the same chart, so that the point $x_j + (t, F^1_j(t))$ is associated with the point
$x_j + (t, F^2_j(t))$ by the vertical line ${(0,s)}_{s \in \mathbb{R}}$ in the chart.  

Recall that
$G_{F^k_j} = \chi_{i,j}(x_j + (t, F^k_j(t))$, for $t \in I_{r_j}$.  Define
\[
\tpsi_j = \psi_1 \circ T^n \circ G_{F^1_j} \circ G_{F^2_j}^{-1} \, . 
\]
The function $\tpsi_j$ is well-defined on $U^2_j$ and $d_0(\tpsi_j , \psi_1 \circ T^n) = 0$.
We can then estimate,
\begin{equation}
\label{eq:unstable second split}
\sum_j \left| \int_{U^1_j} f \, \psi_1 \circ T^n - \int_{U^2_j} f \, \psi_2 \circ T^n \right|
\le \sum_j \left| \int_{U^1_j} f \, \psi_1 \circ T^n - \int_{U^2_j} f \, \tpsi_j \right| +
 \left| \int_{U^2_j} f \, (\tpsi_j - \psi_2 \circ T^n) \right| .
\end{equation}

We estimate the first term on the right side of \eqref{eq:unstable second split} using the strong
unstable norm.  It follows from the uniform hyperbolicity of $T$ and the usual graph transform
arguments (see \cite[Section~4.3]{demers liv}), that
\[
d_{\cW^s}(U^1_j, U^2_j) \le C \Lambda^{-n} \ve \, .
\]
Moreover, by definition $G_{F^1_j}, G_{F^2_j}^{-1} \in \cC^1$ so that by \eqref{eq:holder},
$|\tpsi_j |_{\cC^1(U^2_j)} \le C |\psi_1|_{\cC^1(W^1)}$ for some uniform constant $C$.
Thus,
\begin{equation}
\label{eq:match unstable}
\sum_j \left| \int_{U^1_j} f \, \psi_1 \circ T^n - \int_{U^2_j} f \, \tpsi_j \right| 
\le C \ve^\beta \Lambda^{-\beta n} \| f\|_u \delta_0^{-1} \# \cM_0^n \, ,
\end{equation}
where we have used Lemma~\ref{lem:growth}(b) to sum over the matched pieces since there
is at most one matched piece per element of $\cG_n(W^1)$.

We estimate the second term on the right side of \eqref{eq:unstable second split}
using the strong stable norm,
\[
\sum_j  \left| \int_{U^2_j} f \, (\tpsi_j - \psi_2 \circ T^n) \right|
\le \sum_j \| f \|_s |U^2_j|^{1/p} |\tpsi_j - \psi_2 \circ T^n|_{\cC^\alpha(U^2_j)} \, .
\]
It follows from \cite[Lemma~4.2 and eq. (4.20)]{demers liv} that,
\[
|\tpsi_j - \psi_2 \circ T^n|_{\cC^\alpha(U^2_j)} \le C \ve^{1-\alpha} \, .
\]
Putting this together with the above estimate and summing over $j$ yields,
\[
\sum_j  \left| \int_{U^2_j} f \, (\tpsi_j - \psi_2 \circ T^n) \right|
\le C \ve^{1-\alpha} \| f \|_s \delta_0^{-1} \# \cM_0^n \, .
\]

Finally, collecting the above estimate with \eqref{eq:match unstable} in \eqref{eq:unstable second split}
and adding the estimate over unmatched pieces from \eqref{eq:unmatched}, yields by
\eqref{eq:unstable split},
\[
\left| \int_{W^1} \cL^n f \, \psi_1 - \int_{W^2} \cL^n f \psi_2 \right|
\le C \delta_0^{-1} \big(\ve^\beta \Lambda^{-\beta n} \| f \|_u + \ve^{1-\alpha} \| f\|_s + \ve^{1/p} \kappa^{-n/p} \|f \|_s \big) \# \cM_0^n \, .
\]
Then, since $\beta \le \min \{ 1- \alpha, 1/p \}$ according to \eqref{eq:restrict}, 
we may divide through by $\ve^\beta$, and
take the appropriate suprema to complete the proof of \eqref{eq:unstable norm}.


\subsection{A spectral gap for $\cL$}
\label{sec:spectral}

We prove that $\cL$ has a spectral gap in a series of lemmas, first establishing its
quasi-compactness, Lemma~\ref{lem:radius}, then characterizing elements of its peripheral
spectrum, Lemmas~\ref{lem:peripheral} and \ref{lem:leafwise}, and finally concluding the
existence of a spectral gap, Lemma~\ref{lem:gap}.  These are all the items of Theorem~\ref{thm:spectral}.

\begin{lemma}
\label{lem:radius}
The spectral radius of $\cL$ on $\cB$ is $e^{h_*}$, while its essential spectral radius is at
most $\sigma e^{h_*}$ for any $\sigma > \max \{ \Lambda^{-\beta}, \rho \}$.  Thus $\cL$ is quasi-compact on $\cL$.
Moreover, the peripheral spectrum of $\cL$ contains no Jordan blocks.
\end{lemma}

\begin{proof}
First we establish the upper bound on the spectral radius of $\cL$ using Proposition~\ref{prop:LY}
and Corollary~\ref{cor:upper M}.  Fix $\sigma <1$ such that 
$\sigma > \max \{ \Lambda^{-\beta }, \rho \}$.  Next, choose $N>0$ such that
$C \delta_0^{-2} 2 c_1^{-1} \max \{ \Lambda^{-\beta N}, \rho^N \} \le \frac 12 \sigma^N$.  Finally, choose $c_u >0$ such that $c_u C \delta_0^{-2} 2 c_1^{-1} \kappa^{-N/p} \le \frac 12 \sigma^N$.  Then,
\[
\begin{split}
\| \cL^N f \|_{\cB}  & = \| \cL^N f \|_s + c_u \| \cL^N f \|_u \\
& \le \left( \tfrac 12 \sigma^N \| f \|_s + C \delta_0^{-2} 2 c_1^{-1} \kappa^{-N/p} |f|_w
+ c_u \tfrac 12 \sigma^N \| f \|_u + c_u C \delta_0^{-1} 2 c_1^{-1} \kappa^{-N/p} \| f \|_s \right) e^{N h_*} \\
& \le \left( \sigma^N \| f \|_{\cB} + C' \delta_0^{-2} \kappa^{-N/p} |f|_w \right) e^{N h_*} \, .
\end{split}
\]
This is the standard Lasota-Yorke inequality for $\cL$, which, coupled with the compactness
of the unit ball of $\cB$ in $\cB_w$ (Lemma~\ref{lem:compact}), is sufficient to conclude \cite{hennion} that the essential spectral
radius of $\cL$ is at most $\sigma e^{h_*}$, and its spectral radius is at most $e^{h_*}$.

To prove the lower bound on the spectral radius, we estimate using \eqref{eq:delta_1} and Lemma~\ref{lem:lower}.
Take $W \in \cW^s$ with $|W| \ge \delta_1/3$.  Then for $n \ge n_1$ we have,
\begin{equation}
\label{eq:lower spec}
\begin{split}
\| \cL^n 1\|_{\cB} & \ge \int_W \cL^n 1 \, dm_W  = \sum_{W_i \in \cG_n^{\delta_1}} |W_i|
\ge \sum_{W_i \in L_n^{\delta_1}(W)} \delta_1/3 \\
& \ge \frac{2 \delta_1}{9} \# \cG_n(W) \ge \frac{2 \delta_1}{9} c_0 \# \cM_0^n \, .
\end{split}
\end{equation}
Then taking the limit as $n \to \infty$ and using the definition of $h_*$,
\[
\limsup_{n \to \infty} \frac 1n \log \| \cL^n \|_{\cB}
\ge \limsup_{n \to \infty} \frac 1n \log \big( \| \cL^n 1 \|_{\cB} / \| 1 \|_{\cB} \big)
\ge \limsup_{n \to \infty} \frac 1n \log \big( \# \cM_0^n \big) = h_* \, ,
\]
which proves that  the spectral radius of $\cL$ is at least $e^{h_*}$.  We conclude that the
spectral radius of $\cL$ is in fact $e^{h_*}$ and so $\cL$ is quasi-compact since its 
essential spectral
radius is bounded by $\sigma e^{h_*}$.

Finally, the lack of Jordan blocks stems from Corollary~\ref{cor:upper M} and
Proposition~\ref{prop:LY}, which together imply $\| \cL^n \|_{\cB} \le C e^{n h_*}$ for
all $n \ge 0$.
\end{proof}

Let $\mathbb{V}_\theta$ denote the eigenspace associated to the eigenvalue
$e^{h_* + 2\pi i \theta}$.  Due to the quasi-compactness of $\cL$ and the absence of
Jordan blocks, the spectral projector $\Pi_\theta : \cB \to \mathbb{V}_\theta$ is well-defined
in the uniform topology of $L(\cB, \cB)$ and can be realized as,
\begin{equation}
\label{eq:theta}
\Pi_\theta = \lim_{n \to \infty} \frac 1n \sum_{k=0}^{n-1} e^{-k h_* } e^{-2\pi i \theta k} \cL^k \, .
\end{equation}
Let $\bV = \oplus_{\theta} \bV_\theta$, where the sum is taken over $\theta$ corresponding
to eigenvalues of $\cL$.  Note that $\bV$ is finite dimensional by the quasi-compactness of
$\cL$.
Analogously, and as in the statement of Theorem~\ref{thm:spectral}, define 
\begin{equation}
\label{eq:nu def}
\nu_0 = \Pi_0 1 := \lim_{n \to \infty} \frac 1n \sum_{k=0}^{n-1} e^{-k  h_* } \cL^k 1 \, .
\end{equation}
Since we have proved uniform bounds of the form $\| \cL^k \|_{\cB} \le C e^{kh_*}$, the limit
above exists and satisfies $\cL \nu_0 = e^{h_*} \nu_0$.
A priori, however, $\nu_0$ may be 0 (if $e^{h_*}$ is not in the spectrum of $\cL$).  
The following lemma shows this is not the case, and provides
an important characterization of the peripheral spectrum of $\cL$.

\begin{lemma} (Peripheral spectrum of $\cL$)
\label{lem:peripheral}
\begin{itemize}
  \item[a)]   The distribution $\nu_0 = \Pi_0 1 \neq 0$ is a non-negative Radon measure and $e^{h_*}$ is in the spectrum of $\cL$.  
  \item[b)]  All elements of $\bV$ are signed measures, absolutely continuous
  with respect to $\nu_0$.
  \item[c)]  The spectrum of $e^{-h_*} \cL$ consists of a finite number of cyclic groups; in particular,
  each $\theta$ is rational.  
\end{itemize}
\end{lemma}

\begin{proof}
(a)  By density of $\cC^1(M)$ in $\cB$, since $\mathbb{V}_\theta$ is finite-dimensional,
it follows that $\Pi_\theta \cC^1(M) = \mathbb{V}_\theta$.  Thus for each $\nu \in \bV$, $\nu \neq 0$, there
exists $f \in \cC^1(M)$ such that $\Pi_\theta f = \nu$.  Moreover, for every $\psi \in \cC^1(M)$, we have
\begin{equation}
\label{eq:proj}
| \nu(\psi) | = | \Pi_\theta f(\psi) | \le \lim_{n \to \infty} \frac 1n \sum_{k=0}^{n-1} e^{-h_* k} |\cL^k f(\psi)| 
\le |f|_\infty \Pi_0 1(|\psi|) \, ,
\end{equation}
so that $\Pi_0 1 \neq 0$ since $\nu \neq 0$. 
In particular, $e^{h_*}$ is an eigenvalue of $\cL$.  Moreover, since $\Pi_0 1$ is positive
as an element of $(\cC^1(M))^*$, it follows from \cite[Sect. I.4]{Sch} that $\nu_0 = \Pi_0 1$ is a 
non-negative Radon
measure on $M$.  

\medskip
\noindent
(b)  Applying \eqref{eq:proj} again to $\nu \in \bV_\theta$, we conclude that every
element of $\bV_\theta$ is a signed measure, absolutely continuous with respect to
$\nu_0$.  Moreover,  setting $f_\nu = \frac{d\nu}{d(\nu_0)}$, 
it follows that $f_\nu \in L^\infty(M, \nu_0)$.

\medskip
\noindent
(c)  Suppose $\nu \in \bV_\theta$.  Then using part (b), for any $\psi \in \cC^1(M)$,
\begin{equation}
\label{eq:density}
\begin{split}
\int_M \psi \, f_\nu \, d\nu_0 & = \nu(\psi) = e^{-h_*} e^{-2\pi i \theta} \cL \nu (\psi)
= e^{-h_*} e^{-2\pi i \theta} \nu (\frac{\psi \circ T}{J^sT}) \\
& = e^{-h_*} e^{-2\pi i \theta} \nu_0 (f_\nu \frac{\psi \circ T}{J^sT})
= e^{-h_*} e^{-2\pi i \theta} \cL \nu_0 (\psi f_\nu \circ T^{-1}) \\
& = e^{-2 \pi i \theta} \int_M \psi \, f_\nu \circ T^{-1} \, d\nu_0 \, .
\end{split}
\end{equation}
Thus $f_\nu \circ T^{-1} = e^{2 \pi i \theta} f_\nu$, $\nu_0$-a.e.  Define
$f_{\nu, k} = (f_\nu)^k \in L^\infty(\nu_0)$.  It follows as in \cite[Lemma~5.5]{demers liv},
that $d\nu_k := f_{\nu,k} d\nu_0 \in \cB$ for each $k \in \mathbb{N}$.  Then since
$\cL \nu_k = e^{2\pi i k \theta} \nu_k$, it follows that $e^{2 \pi i k \theta}$ is in the
peripheral spectrum of $\cL$ for each $k$.  By the quasi-compactness of $\cL$, this
set must be finite, and so $\theta$ must be rational.
\end{proof}

We remark that elements of $\cB_w$ can be viewed as both distributions on $M$, as well
as families of {\em leafwise distributions} on stable manifolds as follows 
(cf. \cite[Definition~7.5]{max}).  For
$f \in \cC^1(M)$, the map defined by 
\[
\cK_{(W, f)}(\psi) = \int_W f \psi \, dm_W, \qquad \psi \in \cC^1(W) \, ,
\]
can be viewed as a distribution of order 1 on $W$.  Since $\cK_{(W,f)}(\psi) \le |f|_w |\psi|_{\cC^1(W)}$,
$\cK_{(W, \, \cdot \,)}$ can be extended to $f \in \cB_w$.  We denote this extension by
$\int_W \psi f$, and we call the associated family of distributions the {\em leafwise distribution}
$(f, W)_{W \in \cW^s}$ corresponding to $f$.  If, in addition, $f \in \cB_w$ satisfies
$\int_W \psi f \ge 0$ for all $\psi \ge 0$, then by \cite[Section~I.4]{Sch}, the
leafwise distribution is in fact a leafwise measure.

Recall the disintegration of $\musrb$ used in the proof of Lemma~\ref{lem:embed} into
conditional measures $\musrb^\xi$ on the family of stable manifolds $\cF = \{ W_\xi \}_{\xi \in \Xi}$,
and a factor measure $\hatmusrb$ on the index set $\Xi$.  We have
$d\musrb^\xi = |W_\xi|^{-1} g_\xi dm_{W_\xi}$, where $g_\xi$ is uniformly log-H\"older continuous
by \eqref{eq:log g}.

\begin{lemma}
\label{lem:leafwise}
Let $\nu_0^\xi$ and $\hat \nu_0$ denote the conditional measures on $W_\xi$ and factor measure
on $\Xi$, respectively, obtained by disintegrating $\nu_0$ on the family of stable manifolds
$\cF$.  For all $\psi \in \cC^1(M)$,
\[
\int_{W_\xi} \psi \, d\nu^\xi_0 = \frac{\int_{W_\xi} \psi \, g_\xi \, \nu_0}{\int_{W_\xi} g_\xi \, \nu_0}
\quad \mbox{for all $\xi \in \Xi$, and} \quad
d\hat\nu_0(\xi) = |W_\xi|^{-1} \Big( \int_{W_\xi} g_\xi \, \nu_0 \Big) \, d\hatmusrb(\xi) \, .
\]
Moreover, viewed as a leafwise measure, $\nu_0(W) > 0$ for all $W \in \cW^s$.
\end{lemma}

\begin{proof}
We prove the last claim first.  For $W \in \cW^s$, let $n_2 \le \bar C_2 |\log(|W|/\delta_1)|$
be the constant from the proof of Corollary~\ref{cor:most grow} applied in the case
$\ve = 1/3$ and $\delta_1$ as chosen in \eqref{eq:delta_1}.  Let $V \in \cG_{n_2}^{\delta_1}(W)$
have $|V| \ge \delta_1/3$.  Then using \eqref{eq:delta_1} and Lemma~\ref{lem:lower},
\[
\begin{split}
\int_W \nu_0 & = \lim_{n \to \infty} \frac 1n \sum_{k=0}^{n-1} e^{-k h_*} \int_W \cL^k 1 \, dm_W 
\ge \lim_{n \to \infty} \frac 1n \sum_{k=n_1+n_2}^{n-1} e^{-k h_*} \sum_{W_i \in \cG_{k-n_2}(V)} |W_i| \\
& \ge \lim_{n \to \infty} \frac 1n \sum_{k=n_1+n_2}^{n-1} e^{-k h_*} \tfrac{2\delta_1}{9}  c_0 \# \cM_0^{k-n_2}
= \tfrac{2c_0 \delta_1}{9} e^{-(n_1+n_2)h_*}\lim_{n \to \infty} \sum_{k=0}^\infty e^{-k h_*} \# \cM_0^k \, .
\end{split}
\]
We claim that the last limit cannot be 0.  For suppose it were 0.  Then for any $W \in \cW^s$,
$\psi \in \cC^1(W)$, we would have by Lemma~\ref{lem:growth}(b),
\[
\begin{split}
\int_W \psi \, \nu_0 & = \lim_{n \to \infty} \frac 1n \sum_{k=0}^{n-1} e^{-k h_*} \int_W \psi \, \cL^k 1 \, dm_W \le \lim_{n \to \infty} \frac 1n \sum_{k=0}^{n-1} e^{-k h_*}  \sum_{W_i \in \cG_k(W)} |\psi|_\infty |W_i| \\
& \le  \lim_{n \to \infty} \frac 1n \sum_{k=0}^{n-1} e^{-k h_*} C \# \cM_0^k = 0 \, , 
\end{split}
\]
which would imply $\nu_0 = 0$, a contradiction.  This proves the claim, and recalling the
definition of $n_2$, we conclude that
\begin{equation}
\label{eq:nu positive}
\nu_0(W) \ge C' |W|^{h_* \bar C_2} \qquad \mbox{for all $W \in \cW^s$.}
\end{equation}
With \eqref{eq:nu positive} established, the remainder of the proof follows from the definition of
convergence in the weak norm, precisely as in \cite[Lemma~7.7]{max}.
\end{proof}

We are finally ready to prove the final point of our characterization of the peripheral
spectrum of $\cL$.

\begin{lemma}
\label{lem:gap}
$\cL$ has a spectral gap on $\cB$.
\end{lemma}

\begin{proof}
Recalling Lemma~\ref{lem:peripheral}(c), suppose $\nu_q \in \bV_{p/q}$.  Then
$\cL^q \nu_q = e^{q h_*}\nu_q$ and $\cL^q \nu_0 = e^{q h_*} \nu_0$.  Since $T^q$ is also mixing
and the spectral radius of $\cL^q$ is $e^{q h_*}$, it suffices
to prove that mixing implies the eigenspace corresponding to $e^{h_*}$ is simple in order to 
conclude that $\cL$ can have no other eigenvalues of modulus $e^{h_*}$, i.e. $\cL$ has a
spectral gap.  We proceed to prove this claim.

Suppose $\nu_1 \in \bV_0$.  We will show that $\nu_1 = c \nu_0$ for some constant
$c >0$.  By \eqref{eq:density}, there exists $f_1 \in L^\infty(\nu_0)$ such that
$f_1 \nu_1 = \nu_0$ and 
$f_1 \circ T = f_1$, $\nu_0$-a.e.  Letting 
\[
S_nf_1(x) = \sum_{k=0}^{n-1} f_1 \circ T^k(x) \, ,
\]
it follows that the ergodic average $\frac 1n S_nf_1 = f_1$ for all $n \ge 0$. 
This implies that $f_1$ is constant on stable manifolds.  In addition, since 
by Lemma~\ref{lem:leafwise} and \eqref{eq:nu positive},
the factor measure $\hat \nu_0$ is equivalent to $\hatmusrb$ on the index set $\Xi$, we have that 
$f_1 = f_1 \circ T$ on $\hatmusrb$ a.e. $W_\xi \in \cF$, i.e. 
$f_1 = f_1 \circ T$, $\musrb$-a.e.  By the ergodicity of $\musrb$, $f_1=$ constant
$\musrb$-a.e.  But since this constant value holds on each stable manifold
$W_\xi \in \cF$, using again the equivalence of $\hat\nu_0$ and $\hatmusrb$, we conclude
that $f_1$ is constant $\nu_0$-a.e.
\end{proof}


\section{Construction and Properties of the Measure of Maximal Entropy}
\label{sec:max}

Since $\cL : \cB \to \cB$ has a spectral gap, we may decompose $\cL$ as
\begin{equation}
\label{eq:L decomp}
\cL^n f = e^{n h_*} \Pi_0 f + R^n f \, \mbox{ for any $n \ge 1$, $f \in \cB$},
\end{equation}
where $\Pi_0^2 = \Pi_0$, $\Pi_0 R = R \Pi_0 = 0$ and there exists $\bar \sigma <1$ 
and $C>0$ such that $\| e^{-n h_*} R^n \|_{\cB} \le C \bar \sigma^n$.  Indeed, we may
recharacterize the definition of the spectral projector $\Pi_0$ in \eqref{eq:theta} as,
\[
\Pi_0 f = \lim_{n \to \infty} e^{-n h_*} \cL^n f \, ,
\]
where convergence is in the $\cB$ norm.  Indeed, letting $W \in \cW^s$ with
$|W| \ge \delta_1/3$, we have by Lemma~\ref{lem:growth}(b) and \eqref{eq:nu positive},
\[
\begin{split}
0 < \nu_0(W) & = \lim_{n \to \infty} e^{-n h_*} \int_W \cL^n 1 \, dm_W
= \lim_{n \to \infty} e^{-n h_*} \sum_{W_i \in \cG_n(W)} |W_i| \\
& \le \liminf_{n \to \infty} C e^{-n h_*} \# \cM_0^n \, .
\end{split}
\]
This implies the final limit cannot be 0.  We have proved the following.

\begin{lemma}
\label{lem:lower M}
There exists $\bar c_1>0$ such that $\# \cM_0^n \ge \bar c_1 e^{n h_*}$ for all
$n \ge 1$.
\end{lemma}

Next, consider the dual operator, $\cL^* : \cB^* \to \cB^*$, which also has a spectral
gap.
Recalling our identification of $f \in \cC^1(M)$ with the measure $f d\musrb$ from 
Section~\ref{sec:transfer}, define
\begin{equation}
\label{eq:tnu def}
\tnu_0 := \lim_{n \to \infty} e^{-n h_*} (\cL^*)^n d\musrb \, ,
\end{equation}
where convergence is in the dual norm, $\| \cdot \|_{\cB^*}$.  Clearly, $\tnu_0 \in \cB^*$, and
$\cL^* \tnu_0 = e^{h_*} \tnu_0$.  By the positivity of the operator $\cL^*$, we have
$\tnu_0(f) \ge 0$ for each $f \in \cC^1(M)$ with $f \ge 0$ (recalling $\cC^1(M) \subset \cB$).  Thus again
applying \cite[Section~I.4]{Sch}, we conclude that $\tnu_0$ is a Radon measure on $M$.

Next, defining $f_n = e^{-n h_*} \cL^n 1 \in \cB$ for $n \ge 1$, we have,
\[
\tnu_0(f_n) = \lim_{k \to \infty} e^{-k h_*} \langle f_n, (\cL^*)^k d\musrb \rangle 
=  \lim_{k \to \infty} e^{-k h_*} \langle \cL^k f_n, d\musrb \rangle  \, ,
\]
where $\langle \cdot, \cdot \rangle$ denotes the pairing between an element of
$\cB$ and an element of $\cB^*$.  Then, decomposing $\musrb$ into its
conditional measures $\musrb^\xi$ and factor measure $\hatmusrb$ on
$W_\xi$, $\xi \in \Xi$, as in the proof of Lemma~\ref{lem:embed}, 
and letting $\Xi^{\delta_1} \subset \Xi$ denote the set of indices such that $|W_\xi| \ge \delta_1/3$,
we estimate
\[
\begin{split}
\tnu_0(f_n) & = \lim_{k \to \infty} \int_M f_{n+k} \, d\musrb
= \lim_{k \to \infty} \int_{\Xi} d\hatmusrb(\xi) \, e^{-(n+k)h_*} \int_{W_\xi} \cL^{n+k} 1 \, g_\xi \, dm_{W_\xi} |W_\xi|^{-1} \\
& \ge \lim_{k \to \infty} \int_{\Xi^{\delta_1}} d\hatmusrb(\xi) \, e^{-(n+k) h_*} 
\sum_{W_{\xi, i} \in L_{n+k}^{\delta_1}(W_\xi)} \inf_{W_\xi} g_\xi \; |W_{\xi, i}| |W_\xi|^{-1} \\
& \ge  \lim_{k \to \infty} \int_{\Xi^{\delta_1}} d\hatmusrb(\xi) \, e^{-(n+k) h_*} 
C_g^{-1} \tfrac{2 c_0}{9} \# \cM_0^{n+k}
\ge \hatmusrb(\Xi^{\delta_1}) C_g^{-1} \tfrac{2 c_0}{9} \bar c_1 \, ,
\end{split}
\]
for all $n \ge 1$, where we have used \eqref{eq:delta_1} and Lemma~\ref{lem:lower}
for the second inequality, and Lemma~\ref{lem:lower M} for the third.
Since this lower bound is independent of $n$, we have $\tnu_0(\nu_0) > 0$. 

We can at last formulate the following definition, which is our candidate for the measure of
maximal entropy.

\begin{defin}
\label{def:mu_*}
For $\psi \in \cC^1(M)$, define,
\[
\mu_*(\psi) := \frac{\langle \psi \nu_0 , \tnu_0 \rangle}{\langle \nu_0, \tnu_0 \rangle} \, .
\]
\end{defin}
The measure $\mu_*$ is a probability measure on $M$ due to the positivity of $\nu_0$ and 
$\tnu_0$, and since $\langle \nu_0, \tnu_0 \rangle \neq 0$.
Moreover, $\mu_*(\psi \circ T) = \mu_*(\psi)$ so that $\mu_*$ is an invariant measure for $T$.

We may also characterize the spectral projector $\Pi_0$ in terms of this pairing:
for any $f \in \cB$, it follows from \eqref{eq:L decomp} and \eqref{eq:tnu def} that,
\begin{equation}
\label{eq:projector}
\Pi_0 f = \frac{\langle f, \tnu_0 \rangle}{\langle \nu_0, \tnu_0 \rangle } \nu_0 \, .
\end{equation}

It follows immediately from the spectral gap of $\cL$ that $\mu_*$ has exponential
decay of correlations.

\begin{proposition}
\label{prop:decay}
For all $q>0$, there exist constants $C = C(q)$ and $\gamma = \gamma(q) >0$ such that
for all $\vf, \psi \in \cC^q(M)$,
\[
\left| \int_M \vf \, \psi \circ T^n \, d\mu_* - \int_M \vf \, d\mu_* \int_M \psi \, d\mu_* \right| 
\le C |\vf|_{\cC^q(M)} |\psi|_{\cC^q(M)} e^{-\gamma n}
\quad \mbox{ for all $n \ge 0$.}
\]
\end{proposition}

\begin{proof}
We prove the proposition for $\vf, \psi \in \cC^1(M)$.  The result for $q \in (0,1)$ then follows
by a standard approximation argument.

First we verify that $\psi \circ T^n \tnu_0$ is an element of $\cB^*$ for $\psi \in \cC^1(M)$ and
$n \ge 1$.  
We do this by noting that for any $\psi \in \cC^1(M)$, $\psi \tnu_0 \in \cB^*$ by simply
defining,
\[
\langle f, \psi \tnu_0 \rangle := \langle \psi f, \tnu_0 \rangle
\quad \mbox{ for any $f \in \cB$,}
\]
and the expression on the right is bounded by  $|\psi|_{\cC^1} \| f \|_{\cB} \| \nu_0 \|_{\cB^*}$ by Lemma~\ref{lem:piece}(b), and so the pairing defines a bounded, linear functional on $\cB$,
with norm at most $| \psi|_{\cC^1} \| \tnu_0 \|_{\cB^*}$.
Next, define for $n \ge 1$,
\begin{equation}
\label{eq:dual}
\langle f , \psi \circ T^n \tnu_0 \rangle := \langle e^{-n h_*} \cL^n f, \psi \nu_0 \rangle 
= \langle \psi \, e^{-n h_*} \cL^n f, \nu_0 \rangle \, .
\end{equation}
The expression on the right is bounded by 
\[
\| \psi e^{-n h_*} \cL^n f \|_{\cB} \| \tnu_0 \|_{\cB^*}  
\le |\psi|_{\cC^1(M)} e^{-n h_*} \| \cL^n f \|_{\cB} \| \tnu_0 \|_{\cB^*}
\le C  |\psi|_{\cC^1(M)}  \| f \|_{\cB} \| \tnu_0 \|_{\cB^*} \, ,
\]
where we have used Lemma~\ref{lem:piece}(b) for the first inequality and \eqref{eq:L decomp}
for the second, since in particular, $e^{-n h_*} \| \cL^n f \|_{\cB} \le C$.  Thus \eqref{eq:dual}
defines a bounded, linear functional on $\cB$, so $\psi \circ T^n \tnu_0 \in \cB^*$.

Finally, using Definition~\ref{def:mu_*} and \eqref{eq:dual}, noting that $\vf \nu_0 \in \cB$ by Lemma~\ref{lem:piece}(b), and recalling again \eqref{eq:L decomp}, we write
\[
\begin{split}
 \int_M \vf \, \psi \circ T^n & \, d\mu_*  
=  \frac{\langle \vf  \, \nu_0, \psi \circ T^n \tnu_0 \rangle }{\langle \nu_0, \tnu_0 \rangle}
 = \frac{ \langle e^{-n h_*} \cL^n( \vf \, \nu_0) , \psi \tnu_0 \rangle }{\langle \nu_0, \tnu_0 \rangle} \\
& = \frac{\langle \Pi_0(\vf \nu_0) + e^{-n h_*} R^n(\vf \nu_0), \psi \tnu_0 \rangle }{\langle \nu_0, \tnu_0 \rangle}     
= \frac{\langle \vf \nu_0, \tnu_0 \rangle}{\langle \nu_0, \tnu_0 \rangle} 
\frac{\langle \nu_0, \psi \tnu_0 \rangle}{\langle \nu_0, \tnu_0 \rangle}
+ \frac{ \langle e^{-n h_*} R^n(\vf \nu_0), \psi \tnu_0 \rangle }{\langle \nu_0, \tnu_0 \rangle} \, ,
\end{split}
\]
where we have used \eqref{eq:projector}.
The first term on the right is simply $\int_M \vf \, d\mu_* \int_M \psi \, d\mu_*$.  The second term
is bounded by,
\[
C e^{-n h_*} \| R^n(\vf \nu_0) \|_{\cB} \| \psi \tnu_0 \|_{\cB^*}
\le C' \bar \sigma^n \| \vf \nu_0 \|_{\cB} | \psi |_{\cC^1} \| \tnu_0 \|_{\cB^*}
\le C'' \bar \sigma^n |\vf|_{\cC^1} |\psi|_{C^1} \, , 
\]
where we have used Lemma~\ref{lem:piece}(b) and $C''$ depends on $\| \nu_0 \|_{\cB}$,
$\| \tnu_0 \|_{\cB^*}$, and $\langle \nu_0, \tnu_0 \rangle$.
\end{proof}


\subsection{Hyperbolicity and Ergodicity of $\mu_*$}
\label{sec:hyper}

We begin by showing that $\mu_*$ gives small measure to $\ve$-neighborhoods of the
singularity sets $\cS_n^{\pm}$.

\begin{lemma}
\label{lem:control}
For any $k \in \mathbb{N}$, there exists $C_k >0$ such that 
\[
\mu_*(\cN_\ve(\cS_k^{\pm})) \le C_k \ve^{1/p} \, .
\]
In particular, for any $\gamma > p$ and $k \in \mathbb{N}$, for $\mu_*$-a.e. $x \in M$,
there exists $C>0$ such that
\begin{equation}
\label{eq:approach}
d(T^nx, \cS_k^{\pm}) \ge C n^{-\gamma} \, , \quad \mbox{for all $n \ge 0$.}
\end{equation}
\end{lemma}

\begin{proof}
First we prove the claimed bounds with respect to $\nu_0$ for each $\cS^-_k$, $k \ge 1$.  Let
$1_{k, \ve}$ denote the indicator function of the set $\cN_\ve(\cS^-_k)$.  Since
$\cS_k^-$ comprises finitely many smooth curves, all uniformly transverse to the
stable cone, by Lemma~\ref{lem:piece}(b), $1_{k, \ve} \nu_0 \in \cB$, and as a consequence,
$1_{k, \ve} \nu_0 \in \cB_w$.  We claim that,
\begin{equation}
\label{eq:nu bound}
\nu_0(\cN_\ve(\cS_k^-)) \le C |1_{k, \ve} \nu_0|_w \le C_k \ve^{1/p} \, .
\end{equation}
Indeed, the first inequality follows from Lemma~\ref{lem:embed}.  To prove the second
inequality, let $W \in \cW^s$ and $\psi \in \cC^1(W)$ with $| \psi|_{\cC^1(W)} \le 1$.
Due to the transversality of $\cS_k^-$ with the stable cone, $W \cap \cN_\ve(\cS_k^-)$
comprises at most a finite number $N_k$ of curves, depending only on
$\cS_k^-$ and $\delta_0$, and not on $W$, each having length at most $C\ve$.  Thus,
\[
\int_W  1_{k, \ve} \, \psi \, \nu_0 = \sum_i \int_{W_i}  \psi \, \nu_0
\le \sum_i \| \nu_0 \|_{\cB} |W_i|^{1/p} |\psi|_{\cC^\alpha(W_i)} \le C N_k \ve^{1/p} \, ,
\]
and taking the supremum over $\psi$ and $W$ proves the second inequality in
\eqref{eq:nu bound}.

Next, it follows from \eqref{eq:tnu def} and Lemma~\ref{lem:embed} that
\begin{equation}
\label{eq:weak dual}
|\tnu_0(f)| \le C |f|_w, \qquad \mbox{for all $f \in \cB_w$,}
\end{equation}
so that in fact $\tnu_0 \in \cB_w^* \subset \cB^*$.
Thus for each $k \ge 1$, by \eqref{eq:nu bound},
\[
\mu_*(\cN_\ve(\cS_k^-)) = \frac{\tnu_0(1_{k,\ve} \nu_0)}{\tnu_0(\nu_0)} \le C |1_{k,\ve} \nu_0|_w \le C C_k \ve^{1/p} \, .
\]
To prove the bound for $\cS_k^+$, we use the invariance of $\mu_*$ together with the
fact that $T^{-k}\cS_k^- = \cS_k^+$.  Moreover, we have
$T^k(\cN_\ve(\cS_k^+)) \subset \cN_{C \kappa_+^k \ve}(\cS_k^-)$, where $\kappa_+$
is the maximum rate of expansion in the unstable cone. 

Finally, to prove \eqref{eq:approach}, we fix $\gamma > p$ and estimate for each $k \in \mathbb{N}$,
\[
\sum_{n \ge 1} \mu_*(\cN_{n^{-\gamma}}(\cS_k^{\pm})) \le C_k \sum_{n \ge 1} n^{-\gamma/p} < \infty \, .
\]
Thus by the Borel-Cantelli Lemma, $\mu_*$-a.e. $x \in M$ visits $\cN_{n^{-\gamma}}(\cS_k^{\pm})$
only finitely many times along its orbit, completing the proof of the lemma.
\end{proof}

Lemma~\ref{lem:control} immediately implies the following corollary.

\begin{cor}  
\label{cor:atomic}
The following items establish the hyperbolicity of the measure $\mu_*$.
\begin{itemize}
  \item[a)] For any $\cC^1$ curve $V$ uniformly transverse to the stable cone, there exists
  $C>0$ such that $\nu_0(\cN_{\ve}(V)) \le C\ve$ for all $\ve>0$.
  \item[b)]  The measures $\nu_0$ and $\mu_*$ have no atoms, and $\mu_*(W) = 0$
  for all local stable and unstable manifolds, $W$.
  \item[c)] $\int_M |\log d(x, \cS_1^{\pm})| \, d\mu_* < \infty$.
  \item[d)] $\mu_*$-a.e. $x \in M$ has a stable and an unstable manifold of positive length.
\end{itemize}
\end{cor}

\begin{proof}
The proof follows directly from the control established on the measures of the neighborhoods
of the singularity sets in Lemma~\ref{lem:control}.  The argument follows exactly as
in \cite[Corollary~7.4]{max}.
\end{proof}

With the control established in Lemma~\ref{lem:control}, we may follow the
same arguments as in \cite[Section~7.3]{max} to establish the ergodicity of the measure
$\mu_*$.  Indeed, our control is stronger than the bounds $\mu_*(\cN_\ve(\cS_k^{\pm})) \le C_k |\log \ve|^\gamma$ for some $\gamma>1$ available in \cite{max}, and the H\"older continuity
of our strong norm $\| \cdot \|_u$ is stronger than the logarithmic modulus of continuity
available in \cite{max}.  The key result is establishing the absolute continuity of the unstable
foliation with respect to $\mu_*$.  Given a locally maximal Cantor rectangle $R$, let
$\cW^{s/u}(R)$ be the set of stable/unstable manifolds that cross $D(R)$ completely (see
Section~\ref{sec:lower}).

\begin{proposition}
\label{prop:cont}
Let $R$ be a locally maximal Cantor rectangle with $\mu_*(R)>0$.  Fix $W^0 \in \cW^s(R)$,
and for $W \in \cW^s(R)$, let $\Theta_W : W^0 \cap R \to W \cap R$ denote the
holonomy map sliding along unstable manifolds in $\cW^u(R)$.  Then $\Theta_W$ is
absolutely continuous with respect to $\mu_*$.
\end{proposition}

\begin{proof}
This is \cite[Corollary~7.9]{max}.  Its proof relies on the analogous property of absolute continuity
for $\nu_0$, which in turn follows from the control established by the strong norm and
Lemma~\ref{lem:control}.  The final step in the proof is to show that on each $W \in \cW^s(R)$,
the conditional measure $\mu_*^W$ of $\mu_*$ is equivalent to the leafwise measure $\nu_0$
restricted to $W$,
i.e. there exists $C_W >0$ such that
\begin{equation}
\label{eq:equivalence}
C_W \mu_*^W \le \nu_0|_W \le C_W^{-1} \mu_*^W \, .
\end{equation}
This equivalence of the measures follows from the representation of $\nu_0$ as a family
of leafwise measures given by Lemma~\ref{lem:leafwise} as well as the characterization of
$\mu_*$ via the limit,
\[
\mu_*(\psi) = \tnu_0(\nu)^{-1} \tnu_0(\psi \nu) 
= \tnu_0(\nu_0)^{-1} \lim_{n\to \infty} e^{-n h_*} (\cL^*)^n d\musrb(\psi \nu) \, ,
\]
from \eqref{eq:tnu def}.
\end{proof}

\begin{cor}
\label{cor:ergodic}
The absolute continuity of the unstable holonomy with respect to $\mu_*$ implies the following
additional properties.
\begin{itemize}
  \item[a)]  $(T^n, \mu_*)$ is ergodic for all $n \ge 1$.
  \item[b)]  For any open set $O \subset M$, we have $\mu_*(O) > 0$.
\end{itemize}
\end{cor}

\begin{proof}
a)  Using absolute continuity, one establishes that each Cantor rectangle belongs
to a single ergodic component following the usual Hopf argument \cite[Lemma~7.15]{max}.
Then the ergodicity of $T^n$ follows from the assumption that $T$ is topologically mixing
\cite[Proposition~7.16]{max}.

\smallskip
\noindent
b)  The proof is identical to the proof of \cite[Proposition~7.11]{max}.
\end{proof}


\subsection{Entropy of $\mu_*$}
\label{sec:entropy}

In this section, we prove that the measure-theoretic entropy of $\mu_*$ is $h_*$, by
estimating the measure of dynamically defined Bowen balls for $T^{-1}$.  
Recall the metric $\bar d$ defined in \eqref{eq:bar d}.
For
$n \ge 0$ and $\ve > 0$ and $x \in M$, define
\[
B_n(x, \ve) = \{ y \in M : \bar d(T^{-j}y, T^{-j}x) \le \ve, \, \forall \, 0 \le j \le n \} \, .
\]

\begin{lemma}
\label{lem:bowen}
There exists $C>0$ such that for all $\ve >0$ sufficiently small and all $n \ge 0$, we have\footnote{The extra factor of $n$ in this estimate is due to the fact that we do not assume the dynamical 
refinements of
$\cM_0^1$ are simply connected.  Such an assumption would allow us to eliminate this
factor, as in \cite[Proposition~7.12]{max}.}
\[
\mu_*(B_n(x,\ve)) \le C n e^{-n h_*} \, .
\]
\end{lemma}

\begin{proof}
Fix $x \in M$, $\ve >0$ and $n \ge 0$, and let $1_{n, \ve}^B$ denote the indicator function
of the Bowen ball $B_n(x,\ve)$.  We shall prove
\begin{equation}
\label{eq:goal}
\mu_*(B_n(x,\ve))) =\frac{ \tnu_0(1_{n, \ve}^B \nu_0) }{\tnu_0(\nu_0) } \le C |1_{n, \ve}^B \nu_0|_w \le C n e^{-n h_*} \, ,
\end{equation}
where $C>0$ can be chosen independent of $\ve$.
The first inequality follows from \eqref{eq:weak dual}, once we show
that $1_{n,\ve}^B \nu_0 \in \cB_w$.  
To see this, write
\[
1_{n ,\ve}^B = \prod_{j=0}^n 1_{\cN_\ve(T^{-j}x)} \circ T^{-j} = \prod_{j=0}^n \cL_{\mbox{\tiny SRB}}^j(1_{\cN_{\ve}(T^{-j}x)}) \, ,
\]
where $\cL_{\mbox{\tiny SRB}}$ denotes the transfer operator with respect to $\musrb$.  Since
$\cL_{\mbox{\tiny SRB}}$ preserves $\cB_w$ (and also $\cB$) by \cite{demers liv}, the
claim follows since $1_{\cN_\ve(T^{-j}x)}$ satisfies the assumptions of 
Lemma~\ref{lem:piece}:  $\partial \cN_\ve(T^{-j}x)$ consists of a single circular arc, together
with possibly part of $\partial M$, both of which satisfy the weak transversality condition of that lemma
for $\ve$ sufficiently small.
Applying Lemma~\ref{lem:piece}(b) inductively in $j$ completes the proof of the claim, and
of the first inequality in \eqref{eq:goal}.

Next, since $\nu_0$ is a non-negative leafwise measure by Lemma~\ref{lem:leafwise},
we have $\int_W \psi \, \nu_0 \ge 0$ for all $W \in \cW^s$ and $\psi \ge 0$.  Then since
$|\int_W \psi \, \nu_0| \le \int_W |\psi| \, \nu_0$, we can achieve the supremum in the weak norm
of $\nu_0$
by restricting to test functions $\psi \ge 0$.

Now take $W \in \cW^s$, $\psi \in \cC^1(W)$ with $\psi \ge 0$ and $|\psi|_{\cC^1(W)} \le 1$, and suppose
that $W \cap B_n(x,\ve) \neq \emptyset$.  Then using that $\nu_0$ is an eigenfunction of $\cL$,
\[
\int_W \psi \, 1_{n,\ve}^B \, \nu_0 = \int_W \psi \, 1_{n,\ve}^B \, e^{-n h_*} \cL^n \nu_0
= e^{-n h_*} \sum_{W_i \in \cG_n(W)} \int_{W_i} \psi \circ T^n \, 1_{n,\ve}^B \circ T^n \, \nu_0 \, .
\]
Observe that $1_{n , \ve}^B \circ T^n = 1_{T^{-n}(B_n(x,\ve))}$, and that
\[
T^{-n}(B_n(x,\ve)) = \{ y \in M : \bar d(T^{j-n}x, T^jy) \le \ve, \, \forall \, 0 \le j \le n \} \, .
\]
Thus on each $W_i \in \cG_n(W)$ such that $W_i \cap T^{-n}(B_n(x,\ve)) \neq \emptyset$, the
positivity of $\nu_0$ implies,
\[
\int_{W_i} \psi \circ T^n \, 1_{n,\ve}^B \circ T^n \, \nu_0 \le \nu_0(W_i) \le |\nu_0|_w \, .
\]

It remains to estimate the cardinality of such $W_i$.
Recalling \eqref{eq:bar d}, if $\ve < 10 \, \diam(M)$, and $\bar d(T^{j-n}x, T^jy) \le \ve$, then
$T^{j-n}x$  and $T^jy$ belong to the same set $\overline M_{i_j}^+$ for each $j$.
We would like to conclude that then $T^{-n}(B_n(x,\ve))$ belongs to a single element of
$\cM_0^n$, yet this may fail since both the dynamical refinements of $\cM_0^1$ 
and the local components of $T^{-j}W \subset \overline M_{i_j}^+$ may 
not be connected.  
Figure~\ref{fig} shows an example of how these multiple components may arise due to 
intersections of $\cS^+_j$ with $\cS^-$.

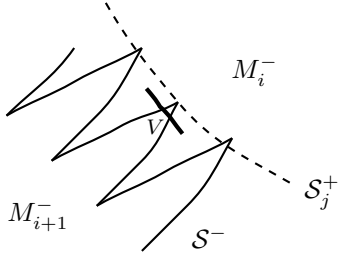
\begin{figure}[ht]
\begin{centering}
\begin{tikzpicture}[x=6mm,y=6mm]

\draw[thick] (0,13) to[out=45, in=235] (1.5,14.5);
\draw[thick] (0,13) to[out=30, in=210] (1.5,13.75) to[out=30, in=210] (3,14.5);
\draw[thick] (1,12) to[out=45, in=225] (2,13) to[out=45, in=240] (3,14.5);
\draw[thick] (1,12) to[out=30, in=210] (2.5,12.75) to[out=30, in=210] (3.8,13.3);
\draw[thick] (2,11) to[out=45, in=225] (3,12) to[out=45, in=240] (3.8,13.3);
\draw[thick] (2,11) to[out=30, in=210] (3.5,11.75) to[out=30, in=210] (5,12.5);
\draw[thick] (3,10) to[out=45, in=225] (4,11) to[out=45, in=240] (5,12.5);

\draw[thick, dashed] (2.2,15.5) to[out=300, in=130] (3.8,13.3) to[out=310, in=150] (6.3, 11.5);
\draw[ultra thick] (3.1,13.6) to[out=310, in=120] (3.4,13.2) to[out=320, in=130] (3.9,12.6);

\node at (3.3,12.75){\scriptsize $V$};
\node at (7,11.3){\small $\cS_j^+$};
\node at (4.5, 10.3){\small $\cS^-$};
\node at (5.5,14){\small $M_i^-$};
\node at (.7,10.8){\small $M_{i+1}^-$};

\end{tikzpicture}
\caption{A possible intersection between $\cS_j^+$ (dashed line) and $\cS^-$ (solid lines).  $\cS^-$ is the boundary between
two domains $M_i^-$ and $M_{i+1}^-$, 
while $\cS_j^+$ is the boundary of elements of $\cM_0^j$.  The 
local stable manifold $V \subset T^{-j}W$ is contained in a single element of $\cM_0^j$, yet
the intersection $V \cap M_i^-$ has two connected components whose images under
$T^{-1}$ will both lie in $M_i^+$ and be within distance
$\ve$ of one another in the metric $\bar d$.  }  
\label{fig}
\end{centering}
\end{figure}

Yet suppose $V \subset V' \in \cG_j(W)$, $|V| < \ve$.  Since $\cS^-$ comprises a finite
number of smooth curves uniformly
transverse to the stable cone, for $\ve$ sufficiently small there can be at most two connected
components of $V$ that lie in the same $M_{i_j}^-$; these will be mapped to the same 
$M_{i_j}^+$ under $T^{-1}$.  Since this subdivision of a set of radius $\ve$ 
can occur at most once per iterate, we have at most $n$ elements
$W_i \in \cG_n(W)$ such that 
$W_i \cap T^{-n}(B_n(x, \ve)) \neq 0$ for $\ve$ sufficiently small.  Putting these estimates
together yields,
\[
\int_W \psi \, 1_{n,\ve}^B \, \nu_0 \le e^{-n h_*} n |\nu_0|_w \, ,
\]
and taking the supremum over $\psi$ and $W$ yields the final inequality in \eqref{eq:goal}.
\end{proof}

\begin{proposition}
\label{prop:entropy}
For $\mu_*$ defined by Definition~\ref{def:mu_*}, we have $h_{\mu_*}(T) = h_*$.
\end{proposition}

\begin{proof}
Recall that $\int_M |\log d(x, \cS_1^{\pm})| \, d\mu_* < \infty$ by Corollary~\ref{cor:atomic}(c),
and that $\mu_*$ is ergodic by Corollary~\ref{cor:ergodic}.  Thus applying
\cite[Proposition~3.1]{DWY},\footnote{Which is a slight modification of the Brin-Katok
local entropy theorem \cite{brin}, applying \cite[Lemma~2]{mane}.  See also 
\cite[Corollary~7.17]{max}.} we conclude that for $\mu_*$-a.e. $x \in M$,
\[
\lim_{\ve \to 0} \liminf_{n \to \infty} - \frac 1n \log \mu_*(B_n(x,\ve)) = 
\lim_{\ve \to 0} \limsup_{n \to \infty} - \frac 1n \log \mu_*(B_n(x,\ve)) = 
h_{\mu_*}(T^{-1} ) = h_{\mu_*}(T) \, .
\]	
On the other hand, Lemma~\ref{lem:bowen} implies that for all $\ve>0$ sufficiently small,
\[
\liminf_{n \to \infty} - \frac 1n \log \mu_*(B_n(x,\ve)) \ge h_* \, .
\]
Thus $h_{\mu_*}(T) \ge h_*$.  But $h_{\mu_*}(T) \le h_*$ by Theorem~\ref{thm:initial}(d), so
equality follows. 
\end{proof}


\subsection{Uniqueness of $\mu_*$}
\label{sec:unique}

In this section we prove that $\mu_*$ is the unique invariant probability
measure with $h_{\mu_*}(T) = h_*$.

The proof of uniqueness follows very closely the proof of uniqueness in
\cite[Section~7.7]{max}.  We include the proof to point out several
differences in the initial estimates on elements of $\cM_{-n}^0$, and for completeness.
The idea of the proof is to adapt Bowen's proof of the uniqueness of 
equilibrium states to the setting of maps with discontinuities.  The key estimates
will be to show that while not all elements of $\cM_{-n}^0$ satisfy good lower
bounds on their measure, most elements (in the sense of Lemma~\ref{lem:most good})
have satisfied good lower bounds at some point in the recent past (in the sense of
Lemma~\ref{lem:long good}).
Recall that $\cM_0^n$ denotes the set of maximal, open connected components on which
$T^n$ is smooth, while $\cM_{-n}^0$ denotes the analogous set for $T^{-n}$.

Choose $\delta_2>0$ sufficiently small that for all $n, k \in \mathbb{N}$, if 
$A \in \cM_{-k}^n$ is such that $\diam^u(A) \le \delta_2$ and $\diam^s(A) \le \delta_2$, then
$A \setminus \cS^{\pm}$ consists of no more that $K_1$ connected components.  Such a choice
of $\delta_2$ is possible by property (P1) and Convention~\ref{convention: n_0=1}.

For $n \ge 1$, define
\[
B_{-2n}^0 = \{ A \in \cM_{-2n}^0 : \forall j, \, 0 \le j \le n/2, \,
T^{-j}A \subset E \in \cM_{-n+j}^0 \mbox{ such that } \diam^u(E) < \delta_2 \} \, .
\]
Define $B_0^{2n} \subset \cM_0^{2n}$ analogously with $\diam^u(E)$ replaced by
$\diam^s(E)$.  Next, let
\begin{equation}
\label{eq:B2n}
B_{2n} := \{ A \in \cM_{-2n}^0 : \mbox{ either $A \in B_{-2n}^0$ or $T^{-2n}A \in B_0^{2n}$ } \} \, , 
\end{equation}
and $G_{2n} = \cM_{-2n}^0 \setminus B_{2n}$.  We think of $B_{2n}$ as the set of `bad' elements
and $G_{2n}$ as the set of `good' elements.

Note that for any $n \ge 1$, each $A \in \cM_{-n}^0$ satisfies $\diam^s(A) \le C \Lambda^{-n}$.
We choose $\bar n \in \mathbb{N}$ such that $C\Lambda^{-\bar n} \le \delta_2$.
Our first lemma shows that the cardinality of $B_{2n}$ is small relative to $e^{2n h_*}$ for
large $n$.

\begin{lemma}
\label{lem:most good}
There exists $C>0$ such that for all $n \ge \bar n$,
\[
\# B_{2n} \le C e^{3n h_*/2} K_1^{n/2} \le C \rho^{n/2} e^{2n h_*}  \, .
\]
\end{lemma}

\begin{proof}
For $n \ge \bar n$, suppose $A \in B_{-2n}^0 \subset \cM_{-2n}^0$.  For simplicity assume
$n$ is even; otherwise, we may use $\lfloor n/2 \rfloor$ in place of $n/2$.
For $0 \le j \le n/2$,
let $A_j$  denote the element of $\cM_{-3n/2 - j}^0$ containing $T^{- (n/2-j)}A \in \cM^{n/2-j}_{-3n/2-j}$.

Since $A \in B_{-2n}^0$ and by choice of $\bar n$, it follows that
$\max \{ \diam^s(A_j), \diam^u(A_j) \} \le \delta_2$ for each $0 \le j \le n/2$.  
By choice of $\delta_2$, the number of connected components of $\cM^1_{-3n/2-j}$ in
each $A_j$ is at most $K_1$.  Fixing $A_0 \in \cM_{-3n/2}^0$ and applying this estimate 
inductively in $j$, we conclude that
$\# \{ A' \in B_{-2n}^0 : T^{-n/2}A' \subset A_0 \} \le K_1^{n/2}$.  Summing over the possible
$A_0 \in \cM_{-3n/2}^0$ yields,
\[
\# B_{-2n}^0 \le \# \cM_{-3n/2}^0 K_1^{n/2} \le C e^{3n h_*/2} \rho^{n/2} \Lambda^{n/2} \le C \rho^{n/2} e^{2n h_*} \, ,
\]
where we have used Proposition~\ref{prop:M0n} and Convention~\ref{convention: n_0=1}
for the second inequality, and Lemma~\ref{lem:growth}(d) for the third.

Next, if $A \in \cM_0^n$, then $\diam^u(A) \le C\Lambda^{-n}$ as well, so the same choice of  
$\bar n$ permits the analogous estimate to hold for $\# B_0^{2n}$ for $n \ge \bar n$.
Finally, since there is a one-to-one correspondence between elements of $\cM_0^n$ and
$\cM_{-n}^0$, we have $\# B_{2n} \le \# B_{-2n}^0 + \# B_0^{2n}$, completing the proof of the
lemma.
\end{proof}

Our next lemma shows that long elements of $\cM_{-j}^0$ enjoy good lower bounds
on their $\mu_*$-measure.  These lower bounds will eventually be linked to elements
of $G_{2n}$.

\begin{lemma}
\label{lem:long good}
There exists a constant $C_{\delta_2} > 0$ such that for all $j \ge 1$ and $A \in \cM_{-j}^0$ such that
$\min \{ \diam^u(A), \diam^s(T^{-j}A) \} \ge \delta_2$,  it follows that, 
\[
\mu_*(A) \ge C_{\delta_2} e^{-j h_*} \, .
\]
\end{lemma}

\begin{proof}
As in the proof of Lemma~\ref{lem:lower}, we choose a finite set
$\cR_{\delta_2} = \{ R_1, \ldots, R_\ell \}$ of locally maximal Cantor rectangles
with $\mu_*(R_i) >0$, such that every stable curve of length $\delta_2$ properly
crosses at least one $R_i$ in the stable direction, and every unstable curve of length 
$\delta_2$ properly crosses at least one $R_i$ in the unstable direction.

Now let $j \ge 1$ and $A \in \cM_{-j}^0$ be as in the statement of the lemma.  By choice
of $\cR_{\delta_2}$, an unstable curve in $A$ properly crosses at least one $R_i \in \cR_{\delta_2}$.
Since $\partial A \subset \cS_n^-$, $\partial A$ cannot intersect any unstable manifolds in
$R_i$ since unstable manifolds cannot be cut under $T^{-n}$.  
Thus $A$ must fully cross $R_i$ in the unstable direction.  Similarly, 
$T^{-j}A \in \cM_0^j$ must fully cross at least one rectangle $R_k \in \cR_{\delta_2}$ in the stable
direction.

Let $\Xi_i$ denote the index set of the family of stable manifolds comprising $R_i$.
If $\xi \in \Xi$, set $W_{\xi, A} = W_\xi \cap A$.  Since $T^{-j}$ is smooth on $A$ and
$T^{-j}A$ fully crosses $R_k$
in the stable direction, it must be that $T^{-j}(W_{\xi,A})$ is a single curve that properly crosses
$R_k$, and so contains a stable manifold in the family corresponding to $R_k$.

Let $s >0$ denote the length of the shortest stable manifold in the rectangles belonging to
$\cR_{\delta_2}$.  Applying \eqref{eq:nu positive}, we estimate for $\xi \in \Xi_i$,
\[
\int_{W_{\xi, A}} \nu_0 = e^{-j h_*} \int_{W_{\xi,A}} \cL^j \nu_0 
= e^{-j h_*} \int_{T^{-j}(W_{\xi,A})} \nu_0 \ge e^{-j h_*} C' s^{h_* \bar C_2} \, .
\]
Next, we let $D(R_i)$ denote the smallest solid rectangle containing $R_i$, and
disintegrate $\mu_*$ on $\{ W_\xi \}_{\xi \in \Xi_i}$ into conditional measures
$\mu_*^\xi$ and a factor measure $\hat \mu_*$ on $\Xi_i$.  Then
using the equivalence of the conditional measure $\mu_*^\xi$ with $\nu_0$ on 
$\mu_*$-a.e. $\xi \in \Xi_i$ from \eqref{eq:equivalence}, we have
\[
\begin{split}
\mu_*(A) & \ge \mu_*(A \cap D(R_i)) \ge \int_{\Xi_i} \mu_*^\xi(A) \, d\hat \mu_*(\xi) \\
& \ge \int_{\Xi_i} C_\xi^{-1} \nu_0(W_{\xi,A}) \, d\hat \mu_*(\xi) 
\ge C' s^{h_* \bar C} e^{-j h_*} \int_{\Xi_i} C_\xi^{-1} \, d\hat \mu_*(\xi) \, , 
\end{split}
\] 
which completes the proof of the lemma due to the finiteness of $\cR_{\delta_2}$.
\end{proof}

Our main proposition of the section is the following.

\begin{proposition}
\label{prop:max}
The measure $\mu_*$ is the unique measure of maximal entropy.
\end{proposition}

\begin{proof}
Since $\mu_*$ is ergodic, it suffices to prove that if $\mu$ is an invariant probability
measure that is singular with respect to $\mu_*$, then $h_\mu(T) < h_{\mu_*}(T)$.

Recall from \eqref{eq:expansive} that with respect to the metric $\bar d$ defined in \eqref{eq:bar d},
$T$ and $T^{-1}$ are expansive:  there exists $\ve_0 >0$ such that if $\bar d(T^jx, T^jy) < \ve_0$
for all $j \in \mathbb{Z}$, then $x=y$.  

For $n \ge 1$, define $\cQ_n$ to be the partition of maximal, connected components of $M$
(with boundary points doubled according to Convention~\ref{con:pointwise}) on which
$T^{-n}$ is continuous.  By the discussion of Section~\ref{sec:ent def}, 
$\cQ_n$ consists of elements with non-empty interior which correspond to elements of
$\cM_{-n}^0$, plus isolated points.  Since the entropy of an atomic measure is 0, we may assume
that $\mu$ gives 0 mass to the isolated points, and it follows from Lemma~\ref{lem:control}
that $\mu_*$ does as well.  Thus the only elements of $\cQ_n$ with positive measure
correspond to elements of $\cM_{-n}^0 = B_n \cup G_n$.  Accordingly, we throw out the atoms in
$\cQ_n$ and continue to call this collection of sets by the same name.

Since $\mu$ is singular with respect to $\mu_*$, there exists a Borel set $F \subset M$ with
$T^{-1}F = F$, $\mu_*(F) = 0$, and $\mu(F)=1$.  Our first step is to approximate $F$
by elements of $\cQ_n$.

\begin{sublem}
\label{sub:diff}
For each $n \ge \bar n$, there exists a finite union $\cC_n$ of elements of $\cQ_n$ such that
\[
\lim_{n \to \infty} (\mu + \mu_*)((T^{-n/2} \cC_n) \bigtriangleup F) = 0 \, .
\]
\end{sublem}

This is \cite[Sublemma~7.24]{max}, and its proof relies on the fact that the diameters of
elements of $T^{-n/2}(\cQ_n)$ tend to 0 as $n$ increases due to the uniform hyperbolicity of $T$.
The invariance of $F$ implies in addition that
\[
\lim_{n \to \infty} (\mu + \mu_*)(\cC_n \bigtriangleup F) = \lim_{n \to \infty} (\mu + \mu_*)((T^{n/2}\cC_n) \bigtriangleup F) = 0 \, .
\]
By the proof of \cite[Sublemma~7.24]{max}, for each $n$, there exists a compact set 
$\cK(n)$ that defines the approximating collection $\tilde \cC_n = T^{-n/2} \cC_n \subset \cM_{-n/2}^{n/2}$, and satisfying
$\cK(n) \nearrow F$ as $n \to \infty$.  To exploit this approximation, we group elements 
$Q \in \cQ_{2n}$ according to whether $T^{-n}Q \subset \cup \tilde \cC_n$ or $T^{-n} Q \cap (\cup \tilde \cC_n) = \emptyset$, where $\cup \tilde \cC_n$ denotes the union of 
elements of $\tilde \cC_n$ in $M$.  Since we have eliminated isolated points,
if $T^{-n}Q \cap (\cup \tilde \cC_n) \neq \emptyset$, then $T^{-n}Q \in \cM_{-n}^n$ is
contained in an element of $\cM_{-n/2}^{n/2}$ that intersects $\cK(n)$.  
Thus $Q \subset \cup T^n \tilde \cC_n = \cup T^{n/2} \cC_n$.

As noted above, the diameters of  $T^{-n}\cQ_{2n}$ tend to 0 as $n \to \infty$, so by the
expansive property of $T$, since the image under $T^{2n}$ of each element of $\cQ_{2n}$ 
is simply connected, $\cQ_{2n}$ is a generating partition for $T^{2n}$ for $n$ large enough.
Thus,
\[
h_\mu(T^{2n}) = h_\mu(T^{2n}, \cQ_{2n}) \le H_\mu(\cQ_{2n}) = - \sum_{Q \in \cQ_{2n}} \mu(Q) \log \mu(Q) \, .
\]
And so,
\[
\begin{split}
2n h_\mu(T) & = h_\mu(T^{2n}) \le - \sum_{Q \in \cQ_{2n}} \mu(Q) \log \mu(Q) \\
& \le - \sum_{Q \subset \cup T^n \tilde \cC^n} \mu(Q) \log \mu(Q) - \sum_{Q \cap (\cup T^n \tilde \cC^n) = \emptyset} \mu(Q) \log \mu(Q) \\
& \le \frac 2e + \mu( \cup T^n \tilde \cC^n) \log \# (\cQ_{2n} \cap T^n\tilde \cC_n) +
\mu(M \setminus (\cup T^n \tilde \cC^n)) \log \#(\cQ_{2n} \setminus (T^n \tilde \cC_n)) \, ,
\end{split}
\]
where in the last line we have used that for $p_j>0$, $\sum_{j=1}^N p_j \le 1$, it holds that
\[
- \sum_{j=1}^N p_j \log p_j \le \frac 1e + (\log N) \sum_{j=1}^N p_j \, ;
\]
see for example \cite[eq.~(20.3.5)]{KH}.  We have applied this fact with $p_j = \mu(Q)$ to both sums separately.  Next, since $- h_{\mu_*}(T) = \left( \mu(\cup T^n\tilde \cC_n) + \mu(M \setminus (\cup T^n \tilde \cC_n)) \right) \log e^{-h_*}$, we estimate for $n \ge \bar n$,
\begin{equation}
\label{eq:splitting C_n}
\begin{split}
2n&(h_\mu(T) - h_{\mu_*}(T)) - \frac 2e \\
& \le \mu(\cup T^n \tilde \cC_n) \log \sum_{Q \subset \cup T^n \tilde \cC_n} e^{-2n h_*}
+ \mu(M \setminus (\cup T^n \tilde \cC_n)) \log \sum_{Q \in \cQ_{2n}\setminus (T^n \tilde \cC_n)}
e^{-2n h_*} \\
& \le \mu(\cup \cC_n) \log \left( \sum_{Q \in G_{2n} \cap T^n \tilde \cC_n} e^{-2n h_*}
+ \sum_{Q \in B_{2n} \cap T^n \tilde \cC_n} e^{-2n h_*} \right) \\
& \qquad + \mu(M \setminus (\cup \cC_n)) \log \left( \sum_{Q \in G_{2n} \setminus T^n \tilde \cC_n} e^{-2n h_*}
+ \sum_{Q \in B_{2n} \setminus T^n \tilde \cC_n} e^{-2n h_*} \right) \, ,
\end{split}
\end{equation}
where for the last inequality, we have used the invariance of $\mu$.  By Lemma~\ref{lem:most good},
the sums over the two subsets of $B_{2n}$ are bounded by $C\rho^{n/2}$.  We focus on estimating
the sums over the two subsets of $G_{2n}$.

The following is proved in \cite[Section~7.7]{max}:  For each $Q \in G_{2n} \subset \cM_{-2n}^0$,
there exists $j, k \in \mathbb{N}$, $0 \le j,k \le n/2$ and $\bar E \in \cM_{-2n+j+k}^0$
such that $T^{-j}Q \subset \bar E$ and $\min \{ \diam^u(\bar E), \diam^s(T^{-2n+j+k}) \} \ge \delta_2$.
We call such a triple $(\bar E, j, k)$ an {\em admissible triple} for $Q \in G_{2n}$,
and note that by Lemma~\ref{lem:long good},
\begin{equation}
\label{eq:bar E}
\mu_*(\bar E) \ge C_{\delta_2} e^{(-2n+j+k)h_*} \, .
\end{equation}

There may be many admissible triples for a fixed $Q \in G_{2n}$.  
Define the unique {\em maximal triple} for $Q$ by taking first the maximum $j$, then the maximum $k$
over all admissible triples for $Q$.

Denote by $\cE_{2n}$ the set of maximal triples corresponding to elements of $G_{2n}$, and
for $(\bar E,j,k) \in \cE_{2n}$, set 
\[
\cA_M(\bar E, j,k) = \{ Q \in G_{2n} : (\bar E, j,k) \mbox{ is the maximal triple for $Q$} \} \, .
\]
Since $\bar E \in \cM_{-2n+j+k}^0$ and $G_{2n} \subset \cM_{-2n}^0$, it follows from Proposition~\ref{prop:M0n} that $\# \cA_M(\bar E, j, k) \le C e^{(j+k)h_*}$ for some $C$ independent of 
$(\bar E, j,k)$ and $n$.

The following sublemma is \cite[Sublemma~7.25]{max}, which implies that if we organize our counting
according to maximal triples, we avoid unwanted redundancies.

\begin{sublem}
\label{sub:disjoint}
If $(\bar E_1, j_1, k_1)$ and $(\bar E_2, j_2, k_2)$ are distinct elements of 
$\cE_{2n}$ with $j_2 \ge j_1$, then $T^{-(j_2-j_1)}\bar E_1 \cap \bar E_2 = \emptyset$.
\end{sublem}

If $Q \in T^n\tilde \cC_n \cap \cA_M(\bar E,j, k)$, then by definition of maximal triple,
$T^{-n+j} \bar E \in \cM_{-n+k}^{n-j}$ contains $T^{-n}Q$.  Since $j, k \le n/2$,
$T^{-n+j}\bar E$ is contained in an element of $\cM_{-n/2}^{n/2}$ that also contains
$T^{-n}Q$ and intersects $\cK(n)$.  Thus $T^{-n+j} \bar E \subset \cup \tilde \cC_n$
whenever $T^n\tilde \cC_n \cap \cA_M(\bar E, j ,k) \neq \emptyset$, and so
$\cA_m(\bar E, j,k) \subset T^n \tilde \cC_n$ whenever 
$T^n\tilde \cC_n \cap \cA_M(\bar E, j ,k) \neq \emptyset$.

Using these observations together with \eqref{eq:bar E}, we estimate
\[
\begin{split}
& \sum_{Q \in G_{2n} \cap T^n \tilde \cC_n} e^{-2n h_*}
\le \sum_{(\bar E, j,k) \in \cE_{2n} : \bar E \subset T^{n-j}\tilde \cC_n} 
\; \sum_{Q \in \cA_M(\bar E,j,k)} e^{-2n h_*} \\
& \le \sum_{(\bar E, j,k) \in \cE_{2n} : \bar E \subset T^{n-j}\tilde \cC_n} C e^{(-2n +j+k)h_*}
\le \sum_{(\bar E, j,k) \in \cE_{2n} : \bar E \subset T^{n-j}\tilde \cC_n} C' \mu_*(\bar E) \\
& \le \sum_{(\bar E, j,k) \in \cE_{2n} : \bar E \subset T^{n-j}\tilde \cC_n} C' \mu_*(T^{-n+j}\bar E)
\le C' \mu_*(\cup \tilde \cC_n) = C' \mu_*(\cup \cC_n) \, ,
\end{split}
\]
where we have used the invariance of $\mu_*$ and the constant $C'$ is independent of $n$.  In the last line we have used Sublemma~\ref{sub:disjoint} in order to sum over the elements of $\cE_{2n}$
without double counting.
Similarly, since $T^{-n+j}\bar E \subset M \setminus \tilde \cC_n$ whenever
$T^n \tilde \cC_n \cap \cA_M(\bar E , j,k) = \emptyset$, the sum over 
$Q \in G_{2n} \setminus T^n \tilde \cC_n$ in \eqref{eq:splitting C_n} 
is bounded by $C' \mu_*(M \setminus (\cup \cC_n))$.

Putting these estimates together with \eqref{eq:splitting C_n} allows us to conclude the argument,
\[
\begin{split}
2n(h_\mu(T) - h_{\mu_*}(T)) - \frac 2e & \le \mu(\cup \cC_n) \log \left(C' \mu_*(\cup \cC_n) + C \rho^{n/2} \right) \\
& \quad + \mu(M \setminus (\cup \cC_n)) \log \left(C' \mu_*(M \setminus (\cup \cC_n)) + C \rho^{n/2}\right) \, .
\end{split}
\]
Then since $\mu(\cup \cC_n) \to 1$ and $\mu_*(\cup \cC_n) \to 0$ as $n \to \infty$, the quantity on the 
right side of the inequality tends to $-\infty$.  This forces $h_\mu(T) < h_{\mu_*}(T)$ to permit
the left side to tend to $-\infty$ as well.
\end{proof}



\end{document}